\documentclass{article}
\usepackage{amsmath}
\usepackage{amsfonts,amssymb,fontenc, mathrsfs, mathtools}
\usepackage{theorem}
\usepackage{tikz}
\usepgflibrary{arrows}
\usetikzlibrary{matrix}

\usepackage[all]{xy}

\input amssym.def

\numberwithin{equation}{section}

\numberwithin{figure}{section}

\makeatletter
\newcommand\qedsymbol{\hbox{$\Box$}}
\newcommand\qed{\relax\ifmmode\Box\else
  {\unskip\nobreak\hfil\penalty50\hskip1em\null\nobreak\hfil\qedsymbol
  \parfillskip=\z@\finalhyphendemerits=0\endgraf}\fi}

\newenvironment{proof}[1][{}]{\par\noindent Proof{#1}. }{\qed}

\newcommand{\bfzero}{{\bf 0}}

\newcommand{\br}{{\}\hspace{-0.07cm}.\hspace{-0.03cm}.\hspace{-0.07cm}\} }}

\newcommand{\Tree}{{\mathsf{T r e e}}}

\newcommand{\Lie}{{\mathsf{Lie}}}
\newcommand{\LLie}{\Lambda^{-1}{\mathsf{Lie}}_{\infty}}

\newcommand{\Com}{{\mathsf{Com}}}

\newcommand{\coLie}{{\mathsf{coLie}}}
\newcommand{\coCom}{{\mathsf{coCom}}}

\newcommand{\Ger}{{\mathsf{Ger}}}

\newcommand{\Bra}{{\mathsf{Braces}}}

\newcommand{\PaCD}{{\mathsf{PaCD}}}

\newcommand{\PaB}{{\mathsf{PaB}}}

\newcommand{\sA}{\mathsf{A}}

\newcommand{\GRT}{{\mathsf{GRT}}}

\newcommand{\Cobar}{{\mathrm{ C o b a r}}}

\newcommand{\Conv}{{\mathrm{Conv}}}
\newcommand{\Ch}{{\mathsf{C h}}}

\renewcommand{\c}{{\circ}}

\newcommand{\MC}{{\mathrm{MC}}}

\newcommand{\mMC}{{\mathfrak{MC}}}

\newcommand{\hotimes}{{\,\hat{\otimes}\,}}

\newcommand{\pb}{\mathrm{pb}}
\newcommand{\VH}{\mathrm{VH}}
\newcommand{\CH}{{\mathrm{CH}}}

\newcommand{\End}{{\mathsf {E n d} }}
\newcommand{\Hom}{{\mathrm{Hom}}}

\newcommand{\DrAssoc}{{\mathrm{DrAssoc}}}
\newcommand{\lab}{{\mathrm{lab}}}

\newcommand{\Act}{{\mathrm{Act}}}
\newcommand{\Id}{{\mathrm{Id}}}
\newcommand{\Harr}{{\mathrm{Harr}}}

\newcommand{\res}{{\mathrm{res}}}

\newcommand{\corr}{{\mathrm{corr}}}

\newcommand{\Cyl}{{\mathrm{Cyl}}}

\newcommand{\Gr}{{\mathrm {Gr}}}

\newcommand{\Der}{{\mathrm {Der}}}

\newcommand{\id}{{\mathsf{id} }}

\newcommand{\Sh}{{\mathrm {S h} }}

\newcommand{\wt}[1]{{\widetilde{#1}}}
\newcommand{\ck}[1]{{\check{#1}}}

\newcommand{\und}[1]{{\underline{#1}}}

\newcommand{\Cbu}{C^{\bullet}}

\newcommand{\al}{{\alpha}}

\newcommand{\bul}{{\bullet}}

\newcommand{\mB}{\mathfrak{B}}

\newcommand{\mH}{{\mathfrak{H}}}
\newcommand{\mT}{{\mathfrak{T}}}

\newcommand{\md}{{\mathfrak{d}}}

\newcommand{\mg}{{\mathfrak{g}}}

\newcommand{\mt}{{\mathfrak{t}}}

\newcommand{\grt}{{\mathfrak{grt}}}
\newcommand{\lie}{{\mathfrak{lie}}}

\newcommand{\Om}{{\Omega}}
\newcommand{\Omb}{{\Omega}^{\bullet}}

\newcommand{\si}{{\sigma}}
\newcommand{\ga}{{\gamma}}

\newcommand{\ve}{{\varepsilon}}
\newcommand{\ka}{{\kappa}}

\newcommand{\cF}{{\cal F}}
\newcommand{\pa}{{\partial}}

\newcommand{\bsi}{{\bf s}^{-1}\,}

\newcommand{\bs}{{\mathbf{s}}}
\newcommand{\bt}{{\mathbf{t}}}

\newcommand{\cC}{{\mathcal C}}

\newcommand{\cM}{{\mathcal M}}

\newcommand{\cD}{{\mathcal{D}}}
\newcommand{\cA}{{\mathcal{A}}}
\newcommand{\cG}{{\mathcal{G}}}

\newcommand{\cT}{{\mathcal{T}}}

\newcommand{\bbK}{{\mathbb K}}

\newcommand{\bbZ}{{\mathbb Z}}
\newcommand{\bbQ}{{\mathbb Q}}

\newcommand{\La}{{\Lambda}}

\newcommand{\te}{\theta}

\newcommand{\D}{{\Delta}}

\newcommand{\sfDel}{{\mathsf{\Delta}}}

\newcommand{\sgn}{{\mathrm {s g n}}}

\DeclareMathOperator{\susp}{\bf{s}}

\date{}

\newtheorem{thm}{Theorem}[section]

\newtheorem{lemma}[thm]{Lemma}
\newtheorem{cor}[thm]{Corollary}

\newtheorem{prop}[thm]{Proposition}

\newtheorem{cond}[thm]{Condition}

\theorembodyfont{\rm}

\newtheorem{pty}[thm]{Property}
\newtheorem{remark}[thm]{Remark}

\title{Tamarkin's construction is equivariant with respect to the action of the Grothendieck-Teichmueller group}

\author{Vasily Dolgushev and Brian Paljug}
\date{}

\begin{document}

\large

\maketitle

\begin{abstract}
Recall that Tamarkin's construction \cite{Hinich}, \cite{Dima-proof} gives us a map from the 
set of Drinfeld associators to the set of homotopy classes of $L_{\infty}$ quasi-isomorphisms 
for Hochschild cochains of a polynomial algebra. Due to results of V. Drinfeld \cite{Drinfeld} 
and T. Willwacher \cite{Thomas} both the source and the target of this map are equipped 
with natural actions of the Grothendieck-Teichmueller group $\GRT_1$.  
In this paper, we use the result from \cite{Action} to prove that 
this map from the set of Drinfeld associators to the set of 
homotopy classes of $L_{\infty}$ quasi-isomorphisms for Hochschild cochains
is $\GRT_1$-equivariant.  
\end{abstract}

\tableofcontents

\section{Introduction}
Let $\bbK$ be a field of characteristic zero, $A = \bbK[x^1, x^2, \dots, x^d]$ be the 
algebra of functions on the affine space $\bbK^d$, and $V_A$ be the algebra of 
polyvector fields on $\bbK^d$.  Let us recall that Tamarkin's construction \cite{Hinich}, 
\cite{Dima-proof} gives us a map from the set of Drinfeld associators to the set of  
homotopy classes of $L_{\infty}$ quasi-isomorphisms from $V_A$ to the 
Hochschild cochain complex $\Cbu(A) : = \Cbu(A,A)$ of $A$.  

In paper \cite{Thomas}, among proving many other things, Thomas Willwacher 
constructed a natural action of the Grothendieck-Teichmueller group $\GRT_1$ 
from \cite{Drinfeld} on the set of homotopy classes of $L_{\infty}$ quasi-isomorphisms 
from $V_A$ to $\Cbu(A)$. On the other hand, it is known \cite{Drinfeld} that the group $\GRT_1$ 
acts simply transitively on the set of Drinfeld associators. 

The goal of this paper is to prove $\GRT_1$-equivariance of the map
resulting from Tamarkin's construction using Theorem 4.3 from \cite{Action}. 
We should remark that the statement about $\GRT_1$-equivariance of Tamarkin's 
construction was made in \cite{Thomas} (see the last sentence of Section 10.2 in 
\cite[Version 3]{Thomas}) in which the author stated that ``it is easy to see''.
The modest goal of this paper is to convince the reader that this statement can 
indeed be proved easily. However, the proof requires an additional tool developed 
in \cite{Action}. 

In this paper, we also prove various statements related to Tamarkin's  
construction \cite{Hinich}, \cite{Dima-proof} which are ``known to specialists'' 
but not proved in the literature in the desired generality.
In fact, even the formulation of the problem of $\GRT_1$-equivariance
of Tamarkin's construction requires some additional work. 

In this paper, Tamarkin's construction is presented in the slightly 
more general setting of graded affine space versus the particular case 
of the usual affine space.  
Thus, $A$ is always the free (graded) commutative algebra over $\bbK$ in 
variables $x^1, x^2, \dots, x^d$ of (not necessarily zero) degrees 
$t_1, t_2, \dots, t_d$, respectively. Furthermore, $V_A$ denotes the 
Gerstenhaber algebra of polyvector fields on the corresponding 
graded affine space, i.e. 
$$
V_A : = S_A \big( \bs\, \Der_{\bbK}(A) \big)\,, 
$$ 
where $\Der_{\bbK}(A)$ denotes the $A$-module of 
derivations of $A$, $\bs$ is the operator which shifts the degree
up by $1$, and $S_A(M)$ denotes the free (graded) commutative 
algebra on the $A$-module $M$.   
  
The paper is organized as follows. 
In Section \ref{sec:Tam-constr}, we briefly review the main part of 
Tamarkin's construction and prove that it gives us a map $\mT$ (see Eq. \eqref{mT})
from the set of homotopy classes of certain quasi-isomorphisms 
of dg operads to the set of homotopy classes of $L_{\infty}$ quasi-isomorphisms 
for Hochschild cochains of $A$.   

In Section \ref{sec:grtact}, we introduce a (prounipotent) group 
which is isomorphic (due to Will\-wa\-cher's theorem \cite[Theorem 1.2]{Thomas}) 
to the prounipotent part $\GRT_1$ of the Grothendieck-Teichm\"uller group  
$\GRT$ introduced in \cite{Drinfeld} by V. Drinfeld. We recall from \cite{Thomas} 
the actions of the group (isomorphic to $\GRT_1$) both 
on the source and the target of the map $\mT$ \eqref{mT}. 
Finally, we prove the main result of this 
paper (see Theorem \ref{thm:main}) which says that Tamarkin's 
map $\mT$ (see Eq. \eqref{mT}) is $\GRT_1$-equivariant.  

In Section \ref{sec:link-to-DrAssoc}, we recall how to use the map $\mT$ 
(see Eq. \eqref{mT} from Sec. \ref{sec:Tam-constr}), a specific solution of 
Deligne's conjecture on the Hochschild complex, and the formality of 
the operad of little discs \cite{Dima-disc} to construct a map 
from the set of Drinfeld associators to the set of homotopy classes of 
$L_{\infty}$ quasi-isomorphisms for Hochschild cochains of $A$. 
Finally, we deduce, from Theorem \ref{thm:main}, $\GRT_1$-equivariance
of the resulting map from the set of Drinfeld associators.
The latter statement (see Corollary \ref{cor:GRT1-equivar} in 
Sec.  \ref{sec:link-to-DrAssoc}) can be deduced from what is written 
in \cite{Thomas} and Theorem \ref{thm:main} given in Section \ref{sec:grtact}. 
However, we decided to add  Section \ref{sec:link-to-DrAssoc} just to make 
the story more complete. 
 
Appendices, at the end of the paper, 
are devoted to proofs of various technical statements 
used in the body of the paper. 

\begin{remark}
\label{rem:admit}
While this paper was in preparation, the 4-th version of preprint 
\cite{Thomas} appeared on arXiv.org. In Remark 10.1 of \cite[Version 4]{Thomas}, 
T. Willwacher gave a sketch of admittedly more economic proof of equivariance 
of Tamarkin's construction with respect to the action of $\GRT_1$. 
\end{remark}

~\\
\noindent
\textbf{Acknowledgements:} We would like to thank Thomas Willwacher 
for useful discussions. We acknowledge the NSF grant DMS-1161867 
for a partial support. Finally, we would like to thank the anonymous referee for 
useful suggestions.

\subsection{Notation and conventions}
The ground field $\bbK$ has characteristic zero.  
For most of algebraic structures considered 
here, the underlying symmetric monoidal category is 
the category $\Ch_{\bbK}$ of unbounded cochain complexes 
of $\bbK$-vector spaces.  We will frequently use the ubiquitous
combination ``dg'' (differential graded) to refer to algebraic objects 
in $\Ch_{\bbK}$\,.  For a cochain complex $V$ we denote 
by $\bs V$ (resp. by $\bs^{-1} V$) the suspension (resp. the 
desuspension) of $V$\,. In other words, 
$$
\big(\bs V\big)^{\bul} = V^{\bul-1}\,,  \qquad
\big(\bs^{-1} V\big)^{\bul} = V^{\bul+1}\,. 
$$
Any $\bbZ$-graded vector space $V$ is tacitly considered 
as the cochain complex with the zero differential. 
For a homogeneous vector $v$ in a cochain complex or 
a graded vector space the notation $|v|$ is reserved for its degree.  

The notation $S_{n}$ is reserved for the symmetric group 
on $n$ letters and  $\Sh_{p_1, \dots, p_k}$ denotes 
the subset of $(p_1, \dots, p_k)$-shuffles 
in $S_n$, i.e.  $\Sh_{p_1, \dots, p_k}$ consists of 
elements $\si \in S_n$, $n= p_1 +p_2 + \dots + p_k$ such that 
$$
\begin{array}{c}
\si(1) <  \si(2) < \dots < \si(p_1),  \\[0.3cm]
\si(p_1+1) <  \si(p_1+2) < \dots < \si(p_1+p_2), \\[0.3cm]
\dots   \\[0.3cm]
\si(n-p_k+1) <  \si(n-p_k+2) < \dots < \si(n)\,.
\end{array}
$$

We tacitly assume the Koszul sign rule. In particular,   
$$
(-1)^{\ve(\si; v_1, \dots, v_m)}
$$ 
will always denote the sign factor corresponding to the permutation $\si \in S_m$ of 
homogeneous vectors $v_1, v_2, \dots, v_m$. Namely, 
\begin{equation}
\label{ve-si-vvv}
(-1)^{\ve(\si; v_1, \dots, v_m)} := \prod_{(i < j)} (-1)^{|v_i | |v_j|}\,,
\end{equation}
where the product is taken over all inversions $(i < j)$ of $\si \in S_m$. 

For a pair $V$, $W$ of $\bbZ$-graded vector spaces we denote by 
$$
\Hom (V,W)
$$
the corresponding inner-hom object in the category of $\bbZ$-graded vector spaces, i.e. 
\begin{equation}
\label{Hom-gr} 
\Hom (V,W) : = \bigoplus_{m} \Hom^m_{\bbK}(V, W)\,,
\end{equation}
where $\Hom^m_{\bbK}(V, W)$ consists of $\bbK$-linear maps $f : V \to W$ such that 
$$
f(V^{\bul}) \subset W^{\bul+m}\,.
$$

For a commutative algebra $B$ and a $B$-module $M$, the notation 
$S_B(M)$ (resp. $\und{S}_B(M)$) is reserved for the symmetric $B$-algebra 
(resp. the truncated symmetric $B$-algebra) on $M$, i.e.
$$
S_B(M) : = B \oplus M \oplus S^2_B(M) \oplus  S^3_B(M) \oplus \dots\,,   
$$  
and 
$$
\und{S}_B(M) : = M \oplus S^2_B(M) \oplus  S^3_B(M) \oplus \dots \,.  
$$

For an $A_{\infty}$-algebra $\cA$, the notation $\Cbu(\cA)$ is reserved for 
the Hochschild cochain complex of $\cA$ with coefficients in $\cA$. 

We denote by $\Com$ (resp. $\Lie$, $\Ger$) the operad governing 
commutative (and associative) algebras without unit 
(resp.  the operad governing Lie algebras, 
Gerstenhaber algebras\footnote{See, for example, Appendix A in \cite{DeligneTw}.} without unit).
Furthermore, we denote by $\coCom$ 
the cooperad which is obtained from $\Com$ by 
taking the linear dual. The coalgebras over $\coCom$ are 
cocommutative (and coassociative) coalgebras without counit. 

The notation $\Cobar$ is reserved for the cobar construction \cite[Section 3.7]{notes}. 

For an operad (resp. a cooperad) $P$ and a cochain complex $V$ we denote by 
$P(V)$ the free $P$-algebra (resp. the cofree\footnote{We tacitly assume that 
all coalgebras are nilpotent.} $P$-coalgebra) generated by $V$: 
\begin{equation}
\label{P-Schur-V}
P(V) : = \bigoplus_{n \ge 0} \Big( P(n) \otimes V^{\otimes \, n} \Big)_{S_n}\,.
\end{equation}
For example,
$$
\Com(V) =  \coCom(V) = \und{S}(V)\,. 
$$

We denote by $\La$ the underlying collection of the 
endomorphism operad 
$$
\End_{\bs\, \bbK}
$$
of the 1-dimensional space $\bs\, \bbK$ placed in degree $1$. 
The $n$-the space of $\La$ is 
$$
\La(n) = \sgn_n \otimes \bs^{1-n}\,,
$$
where $\sgn_n$ denotes the sign representation of the symmetric group $S_n$\,.
Recall that $\La$ is naturally an operad and a cooperad.  

For a (co)operad $P$, we denote 
by $\La P$ the (co)operad which is obtained from 
$P$ by tensoring with $\La$: 
$$
\La P : = \La \otimes P\,.   
$$
It is clear that tensoring with 
$$
\La^{-1} : = \End_{\bs^{-1}\, \bbK}
$$
gives us the inverse of the operation $P \mapsto \La P$\,.

For example, the dg operad $\Cobar(\La\coCom)$ governs 
$L_{\infty}$-algebras and the dg operad 
\begin{equation}
\label{Cobar-La2coCom}
\Cobar(\La^2\coCom)
\end{equation}
governs $\La\Lie_{\infty}$-algebras. 

\subsubsection{$\Ger_{\infty}$-algebras and a basis in $\Ger^{\vee}(n)$}

Let us recall that $\Ger_{\infty}$-algebras  (or homotopy Gerstenhaber 
algebras) are governed by the dg operad 
\begin{equation}
\label{Cobar-Ger-vee}
\Cobar(\Ger^{\vee})\,,
\end{equation}
where $\Ger^{\vee}$ is the cooperad which is obtained by taking 
the linear dual of $\La^{-2} \Ger$. 

For our purposes, it is convenient to introduce the free 
$\La^{-2} \Ger$-algebra $\La^{-2} \Ger (b_1, b_2, \dots, b_n)$ 
in $n$ auxiliary variables $b_1, b_2, \dots, b_n$ of degree $0$ and identify the $n$-th space 
$\La^{-2} \Ger(n)$ of $\La^{-2} \Ger$ with the subspace of 
$\La^{-2} \Ger (b_1, b_2, \dots, b_n)$ spanned by $\La^{-2} \Ger$-monomials
in which each variable $b_j$ appears exactly once. For example, 
$\La^{-2} \Ger(2)$ is spanned by the monomials $b_1 b_2$ and
$\{b_1, b_2\}$ of degrees $2$ and $1$, respectively. 

Let us consider the ordered partitions of the set $\{1, 2, \dots, n\}$ 
\begin{equation}
\label{sp-partition}
\{i_{11}, i_{12}, \dots, i_{1 p_1}\} \sqcup 
\{i_{21}, i_{22}, \dots, i_{2 p_2}\} \sqcup \dots \sqcup \{i_{t1}, i_{t 2}, \dots, i_{t p_t}\}
\end{equation}
satisfying the following properties: 
\begin{itemize}

\item for each $1 \le \beta \le t$ the index $i_{\beta p_{\beta}}$ is 
the biggest among $i_{\beta 1}, \dots, i_{\beta p_{\beta}}$

\item $i_{1 p_1}  <  i_{2 p_2} < \dots <  i_{t p_t}$ (in particular, $i_{t p_t} = n$).

\end{itemize}

It is clear that the monomials 
\begin{equation}
\label{La-2Ger-n-basis}
\{ b_{i_{11}},  \dots, \{ b_{i_{1 (p_1-1)}}, b_{i_{1 p_1}} \br \dots
\{ b_{i_{t1}},  \dots, \{ b_{i_{t (p_t-1)}}, b_{i_{t p_t}} \br
\end{equation}
corresponding to all ordered partitions \eqref{sp-partition} satisfying the above properties
form a basis of the space $\La^{-2} \Ger(n)$\,.

In this paper, we use the notation 
\begin{equation}
\label{dual-basis}
\big( \{ b_{i_{11}},  \dots, \{ b_{i_{1 (p_1-1)}}, b_{i_{1 p_1}} \br \dots
\{ b_{i_{t1}},  \dots, \{ b_{i_{t (p_t-1)}}, b_{i_{t p_t}} \br \big)^*
\end{equation}
for the elements of the dual basis in $\Ger^{\vee}(n) = \big( \La^{-2} \Ger(n) \big)^*$.

\subsubsection{The dg operad $\Bra$} 

In this brief subsection, we recall the dg operad $\Bra$ from \cite[Section 9]{DeligneTw} and 
\cite{K-Soi}\footnote{In paper \cite{K-Soi}, the dg operad $\Bra$ is called the ``minimal operad''.}.

Following \cite{DeligneTw}, we introduce, for every $n \ge 1$, the auxiliary 
set $\cT(n)$. An element of $\cT(n)$ is a planted\footnote{Recall that a {\it planted} tree 
is a rooted tree whose root vertex has valency $1$\,.} planar tree $T$ 
with the following data
\begin{itemize}

\item a partition of the set $V(T)$ of vertices
$$
V(T) = V_{\lab}(T)  \sqcup V_{\nu}(T) \sqcup V_{root}(T)  
$$  
into the singleton $V_{root}(T)$ consisting of the root vertex, 
the set $V_{\lab}(T)$ consisting of $n$ vertices which we call {\it labeled}, and the set $V_{\nu}(T)$
consisting of vertices which we call {\it neutral};

\item a bijection between the set $V_{\lab}(T)$ and the set $\{1,2, \dots, n\}$.

\end{itemize}
We require that each element $T$ of $\cT(n)$ satisfies this condition
\begin{cond}
\label{cond:ge-2-incoming}
Every neutral vertex of $T$ has at least $2$ incoming edges. 
\end{cond}
Elements of $\cT(n)$ are called {\it brace trees}. 

For $n \ge 1$, the vector space $\Bra(n)$ consists of all finite
linear combinations of brace trees in $\cT(n)$. To define a structure of 
a graded vector space on $\Bra(n)$, we declare that each brace 
tree $T \in \cT(n)$ carries degree 
\begin{equation}
\label{deg-T}
|T| = 2 |V_{\nu}(T)| - |E(T)| + 1\,,
\end{equation}
where $|V_{\nu}(T)|$ denotes the total number of neutral vertices 
of $T$ and $|E(T)|$ denotes the total number of edges of $T$. 

Examples of brace trees in $\cT(2)$ (and hence vectors in $\Bra(2)$) are shown on figures 
\ref{fig:Tbrace}, \ref{fig:T21}, \ref{fig:cup}, \ref{fig:cup-opp}.  
\begin{figure}[htp] 
\begin{minipage}[t]{0.5\linewidth}
\centering 
\begin{tikzpicture}[scale=0.5]
\tikzstyle{dec}=[circle, draw, minimum size=5, inner sep=1]
\tikzstyle{nu}=[circle, draw, fill, minimum size=5]
\tikzstyle{root}=[circle, draw, fill, minimum size=0, inner sep=1]
\node[root] (r) at (1, 0) {};
\node [dec] (v1) at (1,1.5) {$1$};
\node [dec] (v2) at (1,3) {$2$};
\draw (r) edge (v1);
\draw (v1) edge (v2);
\end{tikzpicture}
\caption{A brace tree $T \in \cT(2)$} \label{fig:Tbrace}
\end{minipage} 
\begin{minipage}[t]{0.5\linewidth}
\centering 
\begin{tikzpicture}[scale=0.5]
\tikzstyle{dec}=[circle, draw, minimum size=5, inner sep=1]
\tikzstyle{nu}=[circle, draw, fill, minimum size=5]
\tikzstyle{root}=[circle, draw, fill, minimum size=0, inner sep=1]
\node[root] (r) at (1, 0) {};
\node [dec] (v1) at (1,1.5) {$2$};
\node [dec] (v2) at (1,3) {$1$};
\draw (r) edge (v1);
\draw (v1) edge (v2);
\end{tikzpicture}
\caption{A brace tree $T_{21}\in \cT(2)$} \label{fig:T21}
\end{minipage} 
\end{figure}
\begin{figure}[htp]
\begin{minipage}[t]{0.5\linewidth}
\centering 
\begin{tikzpicture}[scale=0.5]
\tikzstyle{dec}=[circle, draw, minimum size=5, inner sep=1]
\tikzstyle{nu}=[circle, draw, fill, minimum size=5]
\tikzstyle{root}=[circle, draw, fill, minimum size=0, inner sep=1]
\node[root] (r) at (1, 0) {};
\node[nu] (v) at (1, 1.5) {};
\node[dec] (v1) at (0,3) {$1$};
\node[dec] (v2) at (2,3) {$2$};
\draw  (r) edge (v);
\draw  (v) edge (v1);
\draw  (v) edge (v2);
\end{tikzpicture}
\caption{A brace tree $T_{\cup}\in \cT(2)$} \label{fig:cup}
\end{minipage} 
~
\begin{minipage}[t]{0.5\linewidth}
\centering 
\begin{tikzpicture}[scale=0.5]
\tikzstyle{dec}=[circle, draw, minimum size=5, inner sep=1]
\tikzstyle{nu}=[circle, draw, fill, minimum size=5]
\tikzstyle{root}=[circle, draw, fill, minimum size=0, inner sep=1]
\node[root] (r) at (1, 0) {};
\node[nu] (v) at (1, 1.5) {};
\node[dec] (v2) at (0,3) {$2$};
\node[dec] (v1) at (2,3) {$1$};
\draw  (r) edge (v);
\draw  (v) edge (v1);
\draw  (v) edge (v2);
\end{tikzpicture}
\caption{A brace tree $T_{\cup^{opp}}\in \cT(2)$} \label{fig:cup-opp}
\end{minipage} 
\end{figure}

According to \eqref{deg-T}, the brace trees $T$ and $T_{21}$
on figures \ref{fig:Tbrace} and \ref{fig:T21}, respectively, carry degree $-1$
and the brace trees $T_{\cup}$, $T_{\cup^{opp}}$ on figures 
\ref{fig:cup}, \ref{fig:cup-opp}, respectively, carry degree $0$. 

Condition \ref{cond:ge-2-incoming} implies that $\cT(1)$ consists of 
exactly one brace tree $T_{\id}$ shown on figure \ref{fig:T-id}. 
\begin{figure}[htp]
\centering
\begin{tikzpicture}[scale=0.5]
\tikzstyle{lab}=[circle, draw, minimum size=5, inner sep=1]
\tikzstyle{n}=[circle, draw, fill, minimum size=5]
\tikzstyle{root}=[circle, draw, fill, minimum size=0, inner sep=1]
\node[root] (r) at (0, 0) {};
\node [lab] (v1) at (0,1.5) {$1$};
\draw (r) edge (v1);
\end{tikzpicture}
\caption{The brace tree $T_{\id}\in \cT(1)$} \label{fig:T-id}
\end{figure} 
Hence we have $\Bra(1) = \bbK.$

Finally, we set $\Bra(0) = \bfzero$. 

For the definition of the operadic multiplications on $\Bra$, 
we refer the reader to\footnote{Strictly speaking $\Bra$ is a suboperad 
of the dg operad defined in \cite[Section 8]{DeligneTw}.}  \cite[Section 8]{DeligneTw} and, in particular, 
Example 8.2. For the definition of the differential on $\Bra$,  
we refer the reader to \cite[Section 8.1]{DeligneTw} and, in particular, 
Example 8.4.   

Let us also recall that the dg operad $\Bra$ acts naturally 
on the Hochschild cochain complex $\Cbu(\cA)$ of any $A_{\infty}$-algebra $\cA$. 
For example, if $T$ (resp. $T_{21}$) is the brace tree shown on 
figure \ref{fig:Tbrace} (resp. figure \ref{fig:T21}), then the expression 
$$
T(P_1, P_2)  + T_{21}(P_1, P_2) \,, \qquad P_1, P_2 \in  \Cbu(\cA)
$$   
coincides (up to a sign factor) with the Gerstenhaber bracket of 
$P_1$ and $P_2$. Similarly, if $T_{\cup}$ is the brace tree shown on figure \ref{fig:cup}, 
then the expression 
$$
T_{\cup}(P_1, P_2)\,, \qquad P_1, P_2 \in  \Cbu(\cA)
$$ 
coincides  (up to a sign factor) with the cup product of $P_1$ and $P_2$. 

For the precise construction of the action of $\Bra$ on  $\Cbu(\cA)$, 
we refer the reader to \cite[Appendix B]{DeligneTw}.

\section{Tamarkin's construction in a nutshell}
\label{sec:Tam-constr}

Various solutions of Deligne's conjecture on the Hochschild cochain complex 
\cite{Bat-Markl}, \cite{BF}, \cite{Swiss}, \cite{K-Soi}, \cite{M-Smith},  \cite{Dima-dg}, 
\cite{Vor} imply that the dg operad $\Bra$ is quasi-isomorphic to the dg operad 
$$
C_{-\bul}(E_2, \bbK)
$$
of singular chains for the little disc operad $E_2$\,.

Combining this statement with the formality \cite{K-mot}, \cite{Dima-disc}
for the dg operad $C_{-\bul}(E_2, \bbK)$, we conclude that the dg operad 
$\Bra$ is quasi-isomorphic to the operad $\Ger$. Hence there exists a 
quasi-isomorphism of dg operads
\begin{equation}
\label{Psi}
\Psi : \Ger_{\infty} \to \Bra
\end{equation}
for which the vector\footnote{Here, we use basis \eqref{dual-basis} in $\Ger^{\vee}(n)$.}   
$\Psi(\bs (b_1 b_2)^*)$ is cohomologous to the sum 
$T + T_{21}$ and the vector $\Psi(\bs \{ b_1, b_2 \}^*)$ is cohomologous to 
$$
\frac{1}{2} (T_{\cup} + T_{\cup^{opp}})\,,
$$
where $T$ (resp. $T_{21}$, $T_{\cup}$, $T_{\cup^{opp}}$) is the brace tree depicted 
on figure \ref{fig:Tbrace} (resp. figure \ref{fig:T21}, \ref{fig:cup}, \ref{fig:cup-opp}).

Replacing $\Psi$ by a homotopy equivalent map we may assume, 
without loss of generality, that 
\begin{equation}
\label{Psi-conditions}
\Psi(\bs (b_1 b_2)^*) =  T + T_{21}\,, 
\qquad 
\Psi(\bs \{ b_1, b_2 \}^*) = \frac{1}{2} (T_{\cup} + T_{\cup^{opp}})\,.
\end{equation} 
So from now on we will assume that the map $\Psi$ \eqref{Psi}
satisfies conditions \eqref{Psi-conditions}. 

Since the dg operad $\Bra$ acts on  the Hochschild cochain complex $\Cbu(\cA)$
of an $A_{\infty}$-algebra $\cA$, the map $\Psi$ equips the Hochschild cochain complex $\Cbu(\cA)$ 
with a structure of a $\Ger_{\infty}$-algebra. We will call it {\it Tamarkin's $\Ger_{\infty}$-structure}
and denote by 
$$
\Cbu(\cA)^{\Psi}
$$
the Hochschild cochain complex of $\cA$ with the $\Ger_{\infty}$-structure coming from $\Psi$.   

The choice of the homotopy class of $\Psi$ \eqref{Psi} (and hence the choice of 
Tamarkin's $\Ger_{\infty}$-structure) is far from unique. In fact, it follows from \cite[Theorem 1.2]{Thomas} 
that, the set of homotopy classes of maps \eqref{Psi} satisfying conditions \eqref{Psi-conditions} form 
a torsor for an infinite dimensional pro-algebraic group. 

A simple degree bookkeeping in $\Bra$ shows that for 
every $n \ge 3$
\begin{equation}
\label{Lie-infty-OK}
\Psi(\bs (b_1 b_2 \dots b_n)^*) = 0\,. 
\end{equation}
Combining this observation with \eqref{Psi-conditions} we see that 
any Tamarkin's $\Ger_{\infty}$-structure on $\Cbu(\cA)$ satisfies the 
following remarkable property: 
\begin{pty}
\label{P:L-infty-OK}
The $\La\Lie_{\infty}$ part of Tamarkin's $\Ger_{\infty}$-structure 
on $\Cbu(\cA)$ coincides with the $\La\Lie$-structure given by 
the Gerstenhaber bracket on $\Cbu(\cA)$. 
\end{pty}

From now on, we only consider the case when $\cA = A$, i.e. 
the free (graded) commutative algebra over $\bbK$ in 
variables $x^1, x^2, \dots, x^d$ of (not necessarily zero) degrees 
$t_1, t_2, \dots, t_d$, respectively. Furthermore, $V_A$ denotes the 
Gerstenhaber algebra of polyvector fields on the corresponding 
graded affine space, i.e. 
$$
V_A : = S_A \big( \bs\, \Der_{\bbK}(A) \big)\,.
$$ 

It is known\footnote{Paper \cite{HKR} treats only the case 
of usual (not graded) affine algebras. However, the proof of 
\cite{HKR} can be generalized to the graded setting in a straightforward manner.} 
\cite{HKR} that the canonical embedding 
\begin{equation}
\label{HKR}
V_A \hookrightarrow \Cbu(A)
\end{equation}
is a quasi-isomorphism of cochain complexes, where $V_A$ is considered 
with the zero differential. In this paper, we refer to \eqref{HKR} as 
the {\it Hochschild-Kostant-Rosenberg embedding}.  

Let us now consider the $\Ger_{\infty}$-algebra $\Cbu(A)^{\Psi}$ for a 
chosen map $\Psi$ \eqref{Psi}. By the first claim of Corollary \ref{cor:exists!} from Appendix \ref{app:rigid}, 
there exists a $\Ger_{\infty}$-quasi-isomorphism
\begin{equation}
\label{U-Ger}
U_{\Ger} : V_A \leadsto \Cbu(A)^{\Psi}
\end{equation} 
whose linear term coincides with the Hochschild-Kostant-Rosenberg embedding. 

Restricting $U_{\Ger}$ to the $\La^2\coCom$-coalgebra 
$$
\La^2 \coCom(V_A)
$$
and taking into account Property \ref{P:L-infty-OK} we get a $\La\Lie_{\infty}$-quasi-isomorphism 
\begin{equation}
\label{U-Lie}
U_{\Lie} :  V_A  \leadsto \Cbu(A)
\end{equation}
of (dg) $\La\Lie$-algebras.  

Thus we deduced the main statement of Tamarkin's construction 
\cite{Dima-proof} which can be summarized as
\begin{thm}[D. Tamarkin, \cite{Dima-proof}]
\label{thm:Tamarkin}
Let $A$ (resp. $V_A$) be the algebra of functions 
(resp. the algebra of polyvector fields) on a graded 
affine space. Let us consider the Hochschild cochain 
complex $\Cbu(A)$ with the standard $\La\Lie$-algebra 
structure. Then, for every map of dg operads $\Psi$ \eqref{Psi}, 
there exists a $\La\Lie_{\infty}$ quasi-isomorphism
\begin{equation}
\label{U-Lie-thm}
U_{\Lie} : V_A \leadsto \Cbu(A)
\end{equation}
which can be extended to a $\Ger_{\infty}$ quasi-isomorphism 
$$
U_{\Ger} : V_A \leadsto \Cbu(A)^{\Psi}
$$
where $V_A$ carries the standard Gerstenhaber algebra structure. \qed
\end{thm}
\begin{remark}
\label{rem:}
In this paper we tacitly assume that the linear part of 
every  $\La\Lie_{\infty}$ (resp. $\Ger_{\infty}$) quasi-isomorphism 
from $V_A$ to $\Cbu(A)$ (resp. $\Cbu(A)^{\Psi}$) coincides with 
the Hochschild-Kostant-Rosenberg embedding of polyvector 
fields into Hochschild cochains. 
\end{remark}

Since the above construction involves several choices it leaves the 
following two obvious questions: 

~\\
{\bf Question A.}  Is it possible to construct two homotopy inequivalent 
$\La\Lie_{\infty}$-quasi-isomorphisms \eqref{U-Lie} corresponding 
to the same map $\Psi$ \eqref{Psi}? And if no then

~\\
{\bf Question B.} Are $\La\Lie_{\infty}$-quasi-isomorphisms $U_{\Lie}$ and $\wt{U}_{\Lie}$  
\eqref{U-Lie} homotopy equivalent if so are the corresponding maps of 
dg operads $\Psi$ and $\wt{\Psi}$ \eqref{Psi}?\\

The (expected) answer (NO) to Question A is given in the following proposition:
\begin{prop}
\label{prop:Lie-morph-unique}
Let $\Psi$ a map of dg operads \eqref{Psi} satisfying 
\eqref{Psi-conditions} and 
\begin{equation}
\label{U-Lie-wtU-Lie}
U_{\Lie}, \, \wt{U}_{\Lie} \,:\, V_A \leadsto \Cbu(A)
\end{equation}
be $\La\Lie_{\infty}$ quasi-morphisms which extend to $\Ger_{\infty}$
quasi-isomorphisms 
\begin{equation}
\label{U-Ger-wtU-Ger}
U_{\Ger}, \, \wt{U}_{\Ger} \,:\, V_A \leadsto \Cbu(A)^{\Psi}
\end{equation}
respectively. Then $U_{\Lie}$ is homotopy equivalent to $\wt{U}_{\Lie}$. 
\end{prop}
\begin{proof}
This statement is essentially a consequence of general 
Corollary \ref{cor:exists!} from Appendix \ref{app:int-formal}. 

Indeed, the second claim of Corollary \ref{cor:exists!} implies that 
$\Ger_{\infty}$-morphisms \eqref{U-Ger-wtU-Ger} are homotopy equivalent. 
Hence so are their restrictions to the $\La^2\coCom$-coalgebra  
$$
\La^2\coCom(V_A)
$$ 
which coincide with $U_{\Lie}$ and $\wt{U}_{\Lie}$, respectively. 
\end{proof}

The expected answer (YES) to Question B is given in the following 
addition to Theorem \ref{thm:Tamarkin}: 
\begin{thm}
\label{thm:indepen}
The homotopy type of $U_{\Lie}$ \eqref{U-Lie} depends only on 
the homotopy type of the map $\Psi$ \eqref{Psi}.
\end{thm}
\begin{proof}
Let $\Psi$ and $\wt{\Psi}$ be maps of dg operads \eqref{Psi} satisfying 
\eqref{Psi-conditions} and let 
\begin{equation}
\label{U-Lie-here}
U_{\Lie}  \,:\, V_A \leadsto \Cbu(A)
\end{equation}
\begin{equation}
\label{wtU-Lie-here}
\wt{U}_{\Lie} \,:\, V_A \leadsto \Cbu(A)
\end{equation}
be $\La\Lie_{\infty}$ quasi-morphisms which extend to $\Ger_{\infty}$
quasi-isomorphisms 
\begin{equation}
\label{U-Ger-and-wtU-Ger}
U_{\Ger} \,:\, V_A \leadsto \Cbu(A)^{\Psi}\,, \qquad 
\textrm{and} \qquad
\wt{U}_{\Ger} \,:\, V_A \leadsto \Cbu(A)^{\wt{\Psi}}
\end{equation}
respectively.  Our goal is to show that if $\Psi$ is homotopy equivalent to $\wt{\Psi}$ then 
$U_{\Lie}$ is homotopy equivalent to $\wt{U}_{\Lie}$.

Let us denote by $\Omb(\bbK)$ the dg commutative algebra of polynomial 
forms on the affine line with the canonical coordinate $t$.  

Since quasi-isomorphisms $\Psi, \widetilde{\Psi} : \Ger_{\infty} \to \Bra$ are 
homotopy equivalent, we have\footnote{For justification of this step see, for example, 
\cite[Section 5.1]{notes}.} a map of dg operads 
\begin{equation}
\label{mH}
\mH : \Ger_{\infty} \to \Bra \otimes \Omb(\bbK)
\end{equation}
such that 
$$
\Psi = p_0 \circ \mH\,, \qquad \textrm{and} \qquad 
\widetilde{\Psi} = p_1 \circ \mH\,,
$$
where $p_0$ and $p_1$ are the canonical maps (of dg operads)
$$
p_0, p_1:  \Bra \otimes \Omb(\bbK) \to \Bra\,,
$$
$$
p_0(v) : = v \Big|_{d t=0,~ t = 0}\,, \qquad
p_1(v) : = v \Big|_{d t=0,~ t = 1}\,.
$$

The map $\mH$ induces a $\Ger_\infty$-structure on $\Cbu(A) \otimes \Omega^\bullet(\bbK)$ such that 
the evaluation maps (which we denote by the same letters)
\begin{equation}
\label{p0p1}
\begin{array}{ccc}
p_0 : \Cbu(A) \otimes \Omega^\bullet(\bbK) \to \Cbu(A)^\Psi\,, &~& p_0(v) : = v \big|_{d t=0,~ t = 0}\,, \\
p_1 : \Cbu(A) \otimes \Omega^\bullet(\bbK) \to \Cbu(A)^{\widetilde{\Psi}}\,, &~&
p_1(v) : = v \big|_{d t=0,~ t = 1}\,.
\end{array}
\end{equation}
are strict quasi-isomorphisms of the corresponding $\Ger_{\infty}$-algebras. 

So, in this proof, we consider the cochain complex $\Cbu(A) \otimes \Omega^\bullet(\bbK)$ 
with the $\Ger_{\infty}$-structure coming from $\mH$ \eqref{mH}.
The same degree bookkeeping argument in $\Bra$ shows 
that\footnote{Here, we use basis \eqref{dual-basis} in $\Ger^{\vee}(n)$.}  
\begin{equation}
\label{Lie-infty-OK-mH}
\mH (\bs (b_1 b_2 \dots b_n)^*) = 0\,. 
\end{equation}
Hence, the $\La\Lie_{\infty}$ part of the $\Ger_{\infty}$-structure 
on $\Cbu(A) \otimes \Omega^\bullet(\bbK)$  coincides with the $\La\Lie$-structure given by 
the Gerstenhaber bracket extended from $\Cbu(A)$ to  $\Cbu(A) \otimes \Omega^\bullet(\bbK)$
to by $ \Omega^\bullet(\bbK)$-linearity.

Since the canonical embedding 
$$
P \mapsto P \otimes 1 :
\Cbu(A) \hookrightarrow \Cbu(A) \otimes \Omega^\bullet(\bbK)
$$ 
is a quasi-isomorphism of cochain complexes, Corollary \ref{cor:exists!} 
from Appendix \ref{app:int-formal} implies that there exists a 
$\Ger_{\infty}$ quasi-isomorphism 
\begin{equation}
\label{U-Ger-mH}
U_{\Ger}^{\mH} : V_A \leadsto \Cbu(A) \otimes \Omega^\bullet(\bbK)\,,
\end{equation}
where $V_A$ is considered with the standard Gerstenhaber structure. 

Since the $\La\Lie_{\infty}$ part of the $\Ger_{\infty}$-structure 
on $\Cbu(A) \otimes \Omega^\bullet(\bbK)$  coincides with the standard 
$\La\Lie$-structure, the restriction of $U_{\Ger}^{\mH}$ to the 
$\La^2\coCom$-coalgebra $\La^2\coCom(V_A)$ gives us 
a homotopy connecting the  $\La\Lie_{\infty}$ quasi-isomorphism 
\begin{equation}
\label{dude-0}
p_0 \circ U_{\Ger}^{\mH} \Big|_{\La^2\coCom(V_A)}   
 \,:\, V_A \leadsto \Cbu(A)
\end{equation}
to the $\La\Lie_{\infty}$ quasi-isomorphism  
\begin{equation}
\label{dude-1}
p_1 \circ U_{\Ger}^{\mH} \Big|_{\La^2\coCom(V_A)}   
 \,:\, V_A \leadsto \Cbu(A)\,,
\end{equation}
where $p_0$ and $p_1$ are evaluation maps \eqref{p0p1}.

Let us now observe that $\La\Lie_{\infty}$ quasi-isomorphisms \eqref{dude-0} and 
\eqref{dude-1} extend to $\Ger_{\infty}$ quasi-isomorphisms 
\begin{equation}
\label{two-dudes}
p_0 \circ U_{\Ger}^{\mH} \,:\, V_A \leadsto \Cbu(A)^{\Psi}\,, \qquad 
\textrm{and} \qquad
p_1 \circ U_{\Ger}^{\mH}  \,:\, V_A \leadsto \Cbu(A)^{\wt{\Psi}}
\end{equation}
respectively. Hence, by Proposition \ref{prop:Lie-morph-unique}, 
$\La\Lie_{\infty}$ quasi-isomorphism \eqref{dude-0} is homotopy equivalent to 
\eqref{U-Lie-here} and $\La\Lie_{\infty}$ quasi-isomorphism \eqref{dude-1}
is homotopy equivalent to \eqref{wtU-Lie-here}.

Thus $\La\Lie_{\infty}$ quasi-isomorphisms \eqref{U-Lie-here} and
\eqref{wtU-Lie-here} are indeed homotopy equivalent. 
\end{proof}

The general conclusion of this section is that Tamarkin's construction \cite{Hinich}, 
\cite{Dima-proof} gives us a map
\begin{equation}
\label{mT}
\mT :  \pi_0 \big( \Ger_{\infty} \to \Bra \big) \to  \pi_0 \big(  V_A \leadsto \Cbu(A)  \big)
\end{equation}
from the set $\pi_0 \big( \Ger_{\infty} \to \Bra \big)$ of homotopy classes of 
operad morphisms \eqref{Psi} satisfying conditions \eqref{Psi-conditions} to 
the set $\pi_0 \big(  V_A \leadsto \Cbu(A)  \big)$ of homotopy classes of 
$\La\Lie_{\infty}$-morphisms from $V_A$ to $\Cbu(A)$ whose linear term 
is the Hochschild-Kostant-Rosenberg embedding.

\section{Actions of $\GRT_1$}
\label{sec:grtact}

Let $\cC$ be a coaugmented cooperad in the category of graded vector spaces 
and $\cC_{\c}$ be the cokernel of the coaugmentation. We assume that 
$\cC(0) = \bfzero$ and $\cC(1) = \bbK$. 

Let us denote by 
\begin{equation}
\label{Der-pr-Cobar-cC}
\Der' \big( \Cobar(\cC) \big)
\end{equation}
the dg Lie algebra of derivation $\cD$ of $\Cobar(\cC)$ satisfying the condition 
\begin{equation}
\label{cond-prime-gen}
p_{\bs\, \cC_{\c}} \circ \cD = 0\,, 
\end{equation}
where $p_{\bs\, \cC_{\c}}$ is the canonical projection $\Cobar(\cC) \to \bs\, \cC_{\c}$.
Conditions $\cC(0) = \bfzero$, $\cC(1) = \bbK$ and \eqref{cond-prime-gen} imply 
that $\Der' \big( \Cobar(\cC) \big)^0$ and $H^0 \big(\Der' ( \Cobar(\cC) ) \big)$
are pronilpotent Lie algebras. 

In this paper, we are mostly interested in the case when $\cC = \La^2\coCom$ and 
$\cC = \Ger^{\vee}$. The corresponding dg operads $\La\Lie_{\infty} = \Cobar(\La^2 \coCom)$ and 
$\Ger_{\infty} = \Cobar(\Ger^{\vee})$ govern $\La\Lie_{\infty}$ and $\Ger_{\infty}$ algebras, 
respectively.  

A simple degree bookkeeping shows that 
\begin{equation}
\label{holes}
\Der' ( \La\Lie_{\infty} )^{\le 0} = \bfzero\,,
\end{equation}
i.e. the dg Lie algebra $\Der' ( \La\Lie_{\infty} )$ does not 
have non-zero elements in degrees $\le 0$. In particular, the 
Lie algebra $H^0 \big(\Der' (\La\Lie_{\infty}) \big)$ is zero. 

On the other hand, the Lie algebra
\begin{equation}
\label{mg}
\mg = H^0 \big(\Der'(\Ger_\infty) \big) 
\end{equation}
is much more interesting. According to Willwacher's theorem \cite[Theorem 1.2]{Thomas}, 
this Lie algebra is isomorphic to the pro-nilpotent part $\grt_1$ of the Grothendieck-Teichm\"uller Lie algebra
$\grt$ \cite[Section 4.2]{AT}. Hence, the group $\exp(\mg)$ is isomorphic to 
the group $\GRT_1 = \exp(\grt_1)$.  
 
Let us now describe how the group $\exp(\mg) \cong \GRT_1$ acts both on 
the source and the target of Tamarkin's map $\mT$ \eqref{mT}.  

\subsection{The action of $\GRT_1$ on $\pi_0 \big( \Ger_{\infty} \to \Bra \big)$} 
Let $v$ be a vector of $\mg$ represented by a (degree zero) cocycle $\cD \in \Der'(\Ger_\infty)$. 
Since the Lie algebra $\Der'(\Ger_\infty)^0$ is pro-nilpotent, $\cD$ gives us an automorphism 
\begin{equation}
\label{exp-cD}
\exp(\cD)
\end{equation}
of the operad $\Ger_{\infty}$. 

Let $\Psi$ be a quasi-isomorphism of dg operads \eqref{Psi}.
Due to Proposition B.2 in \cite{Action}, the homotopy type of the composition 
$$
\Psi \circ \exp(\cD)
$$
does not depend on the choice of the cocycle $\cD$ in 
the cohomology class $v$. Furthermore, for every pair 
of (degree zero) cocycles $\cD, \wt{\cD} \in \Der'(\Ger_\infty)$ 
we have  
$$
\Psi \circ \exp(\cD) \circ \exp(\wt{\cD}) = \Psi \circ \exp \big( \CH(\cD, \wt{\cD}) \big)\,,
$$
where $\CH(x,y)$ denotes the Campbell-Hausdorff series in  symbols $x,y$\,.
 
Thus the assignment 
$$
\Psi \to \Psi \circ \exp(\cD)
$$
induces a {\it right} action of the group $\exp(\mg)$ on the 
set  $\pi_0 \big( \Ger_{\infty} \to \Bra \big)$ of homotopy classes of 
operad morphisms \eqref{Psi}.  

\subsection{The action of $\GRT_1$ on $\pi_0 \big(  V_A \leadsto \Cbu(A)  \big)$} 
Let us now show that $\exp(\mg) \cong \GRT_1$ also acts on 
the set $\pi_0 \big(  V_A \leadsto \Cbu(A)  \big)$ of homotopy classes of 
$\La\Lie_{\infty}$-morphisms from $V_A$ to $\Cbu(A)$. 

For this purpose, we denote by 
\begin{equation}
\label{Ger-infty-End-VA}
\Act_{stan} :  \Ger_{\infty} \to \End_{V_A}
\end{equation}
the operad map corresponding to the standard 
Gerstenhaber structure on $V_A$. 

Then, given a cocycle $\cD \in \Der'(\Ger_\infty)$ representing $v \in \mg$,  
we may precompose map \eqref{Ger-infty-End-VA}
with automorphism \eqref{exp-cD}. 
This way, we equip the graded vector space $V_A$ with a new $\Ger_{\infty}$-structure
$Q^{\exp(\cD)}$ whose binary operations are the standard ones.  Therefore, by Corollary 
\ref{cor:VA-VAQ} from Appendix \ref{app:VA-to-VA}, there exists a $\Ger_{\infty}$ quasi-isomorphism 
\begin{equation}
\label{corr}
U_{\corr} : V_A \to V_A^{Q^{\exp(\cD)}}
\end{equation}
from $V_A$ with the standard Gerstenhaber structure to $V_A$ with 
the $\Ger_{\infty}$-structure $Q^{\exp(\cD)}$.

Due to observation \eqref{holes}, the restriction of $\cD$ onto the suboperad 
$\Cobar(\La^2\coCom) \subset \Cobar(\Ger^{\vee})$ is zero. 
Hence, for every degree zero cocycle $\cD \in \Der'(\Ger_\infty)$, we have
\begin{equation}
\label{aut-to-Lie-infty}
\exp(\cD) \Big|_{\Cobar(\La^2\coCom)} ~ = ~ \Id : \Cobar(\La^2\coCom) \to \Cobar(\La^2\coCom)\,.
\end{equation}
Therefore the $\La\Lie_{\infty}$-part of the 
$\Ger_{\infty}$-structure $Q^{\exp(\cD)}$ coincides with the standard 
$\La\Lie$-structure on $V_A$ given by the Schouten bracket.
Hence the restriction of the $\Ger_{\infty}$ quasi-isomorphism $U_{\corr}$
onto the $\La^2\coCom$-coalgebra $\La^2\coCom(V_A)$ gives us 
a $\La\Lie_{\infty}$-automorphism
\begin{equation}
\label{L-infty-aut}
U^{\cD} : V_A \leadsto V_A\,. 
\end{equation}

Note that, for a fixed  $\Ger_{\infty}$-structure $Q^{\exp(\cD)}$,  
$\Ger_{\infty}$ quasi-isomorphism \eqref{corr} is far 
from unique. However, the second statement of Corollary 
\ref{cor:exists!} implies that the homotopy class 
of \eqref{corr} {\it is} unique. Therefore, the assignment 
$$
\cD \mapsto \big[ U^{\cD}  \big]
$$
is a well defined map from the set of degree zero cocycles of 
$\Der'(\Ger_\infty)$ to homotopy classes of $\La\Lie_{\infty}$-automorphisms
of $V_A$. 
 
This statement can be strengthened further:
\begin{prop}
\label{prop:U-cD-OK}
The homotopy type of $U^{\cD}$ does not depend on the 
choice of the representative $\cD$ of the cohomology class $v$. 
Furthermore, for any pair of degree zero cocycles 
$\cD_1, \cD_2 \in  \Der'(\Ger_\infty)$, the composition 
$U^{\cD_1} \circ U^{\cD_2}$ is homotopy equivalent to 
$U^{\CH(\cD_1, \cD_2)}$, where $\CH(x,y)$ denotes the 
 Campbell-Hausdorff series in  symbols $x,y$. 
\end{prop}
Let us postpone the technical proof of Proposition \ref{prop:U-cD-OK} 
to Subsection \ref{sec:U-cD-OK} and observe that this proposition implies 
the following statement:
\begin{cor}
\label{cor:U-cD-OK}
Let $\cD$ be a degree zero cocycle in $\Der'(\Ger_\infty)$ representing 
a cohomology class $v \in \mg$ and let $U_{\Lie}$ be a $\La\Lie_{\infty}$ quasi-isomorphism 
from $V_A$ to $\Cbu(A)$. The assignment 
\begin{equation}
\label{action-on-morph}
U_{\Lie} \mapsto U_{\Lie} \circ U^{\cD}
\end{equation}
induces a right action of the group $\exp(\mg)$ on the set 
$\pi_0 \big(  V_A \leadsto \Cbu(A)  \big)$ of homotopy classes of 
$\La\Lie_{\infty}$-morphisms from $V_A$ to $\Cbu(A)$. \qed 
\end{cor}

From now on, by abuse of notation, we denote by $U^{\cD}$
any representative in the homotopy class of $\La\Lie_{\infty}$-automorphism 
\eqref{L-infty-aut}.
 
\subsection{The theorem on $\GRT_1$-equivariance}  
The following theorem is the main result of this paper: 
\begin{thm}
\label{thm:main}
Let $\pi_0 \big( \Ger_{\infty} \to \Bra \big)$ be the set of homotopy 
classes of operad maps \eqref{Psi} from the dg operad $\Ger_{\infty}$ governing 
homotopy Gerstenhaber algebras to the dg operad $\Bra$ of brace trees.
Let $\pi_0 \big(  V_A \leadsto \Cbu(A) \big)$ be the set of homotopy classes 
of $\La\Lie_{\infty}$ quasi-isomorphisms\footnote{We tacitly assume that  
operad maps \eqref{Psi} satisfies conditions \eqref{Psi-conditions}
and $\La\Lie_{\infty}$ quasi-isomorphisms $V_A \leadsto \Cbu(A)$ 
extend the Hochschild-Kostant-Rosenberg embedding.} 
from the algebra $V_A$ of polyvector 
fields to the algebra $\Cbu(A)$ of Hochschild cochains of a graded affine space. 
Then Tamarkin's map $\mT$ \eqref{mT} commutes with the action of the 
group $\exp(\mg)$ which corresponds to Lie algebra \eqref{mg}.  
\end{thm}
\begin{proof}
Following \cite[Section 3]{Action}, \cite{Fresse}, we will denote by 
$\Cyl(\Ger^{\vee})$ the $2$-colored dg operad
whose algebras are pairs $(V,W)$ with the data 
\begin{enumerate}
\item a $\Ger_{\infty}$-structure on $V$, 

\item a $\Ger_{\infty}$-structure on $W$, and 

\item a $\Ger_{\infty}$-morphism $F$ from $V$ to $W$, i.e. 
a homomorphism of corresponding dg $\Ger^{\vee}$-coalgebras 
$\Ger^{\vee}(V) \to \Ger^{\vee}(W)$.  
\end{enumerate}

In fact, if we forget about the differential, then 
the operad $\Cyl(\Ger^{\vee})$ is a free operad on a certain 
$2$-colored collection $\cM(\Ger^{\vee})$ naturally associated to $\Ger^{\vee}$. 

Let us denote by 
\begin{equation}
\label{Der-pr-Cyl}
\Der'(\Cyl(\Ger^{\vee})) 
\end{equation}
the dg Lie algebra of derivations $\cD$ of $\Cyl(\Ger^{\vee})$ subject to the 
condition\footnote{It is condition \eqref{prime-cond} which guarantees that 
any degree zero cocycle in $\Der'(\Cyl(\Ger^{\vee}))$ can be exponentiated to 
an automorphism of $\Cyl(\Ger^{\vee})$\,.}
\begin{equation}
\label{prime-cond}
p\circ \cD = 0\,,
\end{equation}
where $p$ is the canonical projection from $\Cyl(\Ger^{\vee})$ onto $\cM(\Ger^{\vee})$. 

The restrictions to the first color part and 
the second color part of  $\Cyl(\Ger^{\vee})$, respectively, 
give us natural maps of dg Lie algebras
\begin{equation}
\label{restrictions}
\res_1,~ \res_2 : \Der'(\Cyl(\Ger^{\vee})) \to \Der'(\Ger_{\infty}), 
\end{equation}
and, due to \cite[Theorem 4.3]{Action}, $\res_1$ and $\res_2$  are chain homotopic 
quasi-isomorphisms.  

Therefore, for every $v \in \mg$ there exists a degree zero cocycle
\begin{equation}
\label{cD-needed}
\cD \in  \Der'(\Cyl(\Ger^{\vee}))
\end{equation} 
such that both $\res_1(\cD)$ and $\res_2(\cD)$ represent the cohomology class $v$. 

Let 
\begin{equation}
\label{U-Ger-start}
U_{\Ger} : V_A \leadsto \Cbu(A)^{\Psi}
\end{equation}
be a $\Ger_{\infty}$-morphism from $V_A$ to $\Cbu(A)$
which restricts to a $\La\Lie_{\infty}$-morphism 
\begin{equation}
\label{U-Lie-start}
U_{\Lie} : V_A \to \Cbu(A)\,.
\end{equation}

The triple consisting of 
\begin{itemize}

\item the standard Gerstenhaber structure on $V_A$, 

\item the $\Ger_{\infty}$-structure on $\Cbu(A)$ coming from a map $\Psi$, and

\item $\Ger_{\infty}$-morphism \eqref{U-Ger-start}

\end{itemize}
gives us a map of dg operads 
\begin{equation}
\label{U-Cyl}
U_{\Cyl} : \Cyl(\Ger^{\vee}) \to \End_{V_A, \Cbu(A)}
\end{equation}
from $\Cyl(\Ger^{\vee})$ to the $2$-colored endomorphism 
operad $\End_{V_A, \Cbu(A)}$ of the pair $(V_A, \Cbu(A))$.
  
Precomposing $U_{\Cyl}$ with the endomorphism 
$$
\exp(\cD) : \Cyl(\Ger^{\vee})  \to \Cyl(\Ger^{\vee}) 
$$
we get another operad map 
\begin{equation}
\label{U-Cyl-exp-cD}
U_{\Cyl}  \circ \exp(\cD) : \Cyl(\Ger^{\vee}) \to \End_{V_A, \Cbu(A)}
\end{equation}
which corresponds to the triple consisting of 

\begin{itemize}

\item the new $\Ger_{\infty}$-structure $Q^{\exp(\res_1(\cD))}$ on 
$V_A$,

\item the $\Ger_{\infty}$-structure on $\Cbu(A)$ corresponding to 
$\Psi \circ \exp(\res_2(\cD))$, and 

\item a $\Ger_{\infty}$ quasi-isomorphism 
\begin{equation}
\label{U-Ger-new}
\wt{U}_{\Ger} : V_A^{Q^{ \exp(\res_1(\cD))} } \leadsto \Cbu(A)^{\Psi \, \circ\, \exp(\res_2(\cD))}
\end{equation}

\end{itemize}

Due to technical Proposition \ref{prop:morph-Lie-OK} proved in Appendix \ref{app:derivations} 
below, the restriction of the $\Ger_{\infty}$ quasi-isomorphism $\wt{U}_{\Ger}$ \eqref{U-Ger-new}
to $\La^2\coCom(V_A)$ gives us \und{the same} $\La\Lie_{\infty}$-morphism 
\eqref{U-Lie-start}. 

On the other hand, by Corollary 
\ref{cor:VA-VAQ} from Appendix \ref{app:VA-to-VA}, there exists a $\Ger_{\infty}$ quasi-isomorphism 
\begin{equation}
\label{corr-here}
U_{\corr} : V_A \to V_A^{Q^{ \exp(\res_1(\cD))} }
\end{equation}
from $V_A$ with the standard Gerstenhaber structure to $V_A$ with
the new $\Ger_{\infty}$-structure $Q^{\exp(\res_1(\cD))}$\,.
 
Thus, composing $U_{\corr}$ with $\wt{U}_{\Ger}$ \eqref{U-Ger-new}, 
we get a $\Ger_{\infty}$ quasi-isomorphism 
\begin{equation}
\label{U-Ger-final}
U^{\exp(\cD)}_{\Ger} : V_A \leadsto \Cbu(A)^{\Psi \, \circ\, \exp(\res_2(\cD))}
\end{equation}
from $V_A$ with the standard Gerstenhaber structure to $\Cbu(A)$ with
the $\Ger_{\infty}$-structure coming from $\Psi  \circ \exp(\res_2(\cD))$.  

The restriction of this $\Ger_{\infty}$-morphism $U^{\exp(\cD)}_{\Ger}$
to $\La^2 \coCom(V_A)$ gives us the $\La\Lie_{\infty}$-morphism 
\begin{equation}
\label{U-Lie-new}
U_{\Lie} \circ U^{\res_1(\cD)}
\end{equation}
where $U^{\res_1(\cD)}$ is the $\La\Lie_{\infty}$-automorphism 
of $V_A$ obtained by restricting \eqref{corr-here} to $\La^2 \coCom(V_A)$\,.

Since both cocycles $\res_1(\cD)$ and $\res_2(\cD)$ of $\Der'(\Ger_{\infty})$ 
represent the same cohomology class $v \in \mg$, Theorem \ref{thm:main} follows.
\end{proof}

\subsection{The proof of Proposition \ref{prop:U-cD-OK}}
\label{sec:U-cD-OK}

Let $\cD$ and $\wt{\cD}$ be two cohomologous cocycles in 
$\Der'(\Ger_{\infty})$ and let $Q^{\exp(\cD)}$, $Q^{\exp(\wt{\cD})}$
be $\Ger_{\infty}$-structures on $V_A$ corresponding to the operad maps
\begin{equation}
\label{Act-exp-cD}
\Act_{stan} \circ \exp(\cD)  : \Ger_{\infty} \to \End_{V_A}\,,
\end{equation}
\begin{equation}
\label{Act-exp-wt-cD}
\Act_{stan} \circ \exp(\wt{\cD}) : \Ger_{\infty} \to \End_{V_A}\,,
\end{equation}
respectively. Here $\Act_{stan}$ is the map $\Ger_{\infty} \to \End_{V_A}$
corresponding to the standard Gerstenhaber structure on $V_A$.

Due to Proposition B.2 in \cite{Action}, operad maps \eqref{Act-exp-cD}
and \eqref{Act-exp-wt-cD} are homotopy equivalent. Hence there exists 
a $\Ger_{\infty}$-structure $Q_t$ on $V_A \otimes  \Omb(\bbK)$ such that 
the evaluation maps 
\begin{equation}
\label{p0p1-VA}
\begin{array}{ccc}
p_0 : V_A \otimes  \Omb(\bbK)  \to V_A^{Q^{\exp(\cD)}} \,, &~& p_0(v) : = v \big|_{d t=0,~ t = 0}\,, \\
p_1 : V_A \otimes \Omega^\bullet(\bbK) \to V_A^{Q^{\exp(\wt{\cD})}} \,, &~&
p_1(v) : = v \big|_{d t=0,~ t = 1}\,.
\end{array}
\end{equation}
are strict quasi-isomorphisms of the corresponding $\Ger_{\infty}$-algebras. 

Furthermore, observation \eqref{holes} implies that the restriction 
of a homotopy connecting the automorphisms $\exp(\cD)$ and $\exp(\wt{\cD})$ of 
$\Ger_{\infty}$ to the suboperad $\La\Lie_{\infty}$ coincides with 
the identity map on  $\La\Lie_{\infty}$ for every $t$. Therefore, the 
$\La\Lie_{\infty}$-part of the $\Ger_{\infty}$-structure $Q_t$ on $V_A \otimes  \Omb(\bbK)$
coincides with the standard $\La\Lie$-structure given by the Schouten bracket. 
 
Since tensoring with $\Omb(\bbK)$ does not change cohomology, 
Corollary \ref{cor:exists!} from Appendix \ref{app:int-formal} implies that 
the canonical embedding $V_A \hookrightarrow V_A \otimes  \Omb(\bbK)$
can be extended to a $\Ger_{\infty}$ quasi-isomorphism 
\begin{equation}
\label{corr-homot}
U^{\mH}_{\corr} : V_A \leadsto  V_A \otimes  \Omb(\bbK)
\end{equation}
from $V_A$ with the standard Gerstenhaber structure to 
$V_A \otimes  \Omb(\bbK)$ with the $\Ger_{\infty}$-structure $Q_t$.

Since the $\La\Lie_{\infty}$-part of the $\Ger_{\infty}$-structure $Q_t$ on $V_A \otimes  \Omb(\bbK)$
coincides with the standard $\La\Lie$-structure given by the Schouten bracket, the restriction 
of $U^{\mH}_{\corr}$ onto $\La^2\coCom(V_A)$ gives us a homotopy 
connecting the $\La\Lie_{\infty}$-automorphisms 
\begin{equation}
\label{zero}
p_0 \circ U^{\mH}_{\corr} ~\Big|_{ \La^2\coCom(V_A)}~ : V_A \leadsto V_A
\end{equation}
and 
\begin{equation}
\label{one}
p_1 \circ U^{\mH}_{\corr} ~\Big|_{ \La^2\coCom(V_A)}~ : V_A \leadsto V_A\,.
\end{equation}

Due to the second part of Corollary \ref{cor:exists!},  $\La\Lie_{\infty}$-automorphism
\eqref{zero} is homotopy equivalent to $U^{\cD}$ and $\La\Lie_{\infty}$-automorphism 
\eqref{one} is homotopy equivalent to $U^{\wt{\cD}}$. 

Thus the homotopy type of $U^{\cD}$ is indeed independent of the representative $\cD$
of the cohomology class.  

To prove the second claim of Proposition \ref{prop:U-cD-OK}, 
we will need to use the 2-colored dg operad $\Cyl(\Ger^{\vee})$ 
recalled in the proof of Theorem \ref{thm:main} above. Moreover, we need 
\cite[Theorem 4.3]{Action} which implies that restrictions \eqref{restrictions} 
are homotopic quasi-isomorphisms of cochain complexes. 

Let $\cD_1$ and $\cD_2$ be degree zero cocycles in $\Der'(\Ger_{\infty})$
and let  $Q^{\exp(\cD_1)}$ be the $\Ger_{\infty}$-structure on $V_A$ which 
comes from the composition 
\begin{equation}
\label{Act-exp-cD-1}
\Act_{stan} \circ \exp(\cD_1)  : \Ger_{\infty} \to \End_{V_A}\,,
\end{equation}
where $\Act_{stan}$ denotes the map $\Ger_{\infty} \to \End_{V_A}$
corresponding to the standard Gerstenhaber structure on $V_A$.  

Let $U_{\Ger, 1}$ be a $\Ger_{\infty}$-quasi-isomorphism 
\begin{equation}
\label{U-Ger-1}
U_{\Ger, 1} : V_A \leadsto V_A^{Q^{\exp(\cD_1)} }\,,
\end{equation}
where the source is considered with the standard Gerstenhaber structure. 

By construction, the $\La\Lie_{\infty}$-automorphism 
$$
U^{\cD_1} : V_A \leadsto V_A
$$
is the restriction of $U_{\Ger, 1}$ onto $\La^2 \coCom(V_A)$. 

Let us denote by $U_{\Cyl}^{V_A}$ the operad map 
$$
U_{\Cyl}^{V_A} :  \Cyl(\Ger^{\vee}) \to \End_{V_A, V_A} 
$$
which corresponds to the triple: 
\begin{itemize}

\item the standard Gerstenhaber structure on the first copy of $V_A$, 

\item the $\Ger_{\infty}$-structure  $Q^{\exp(\cD_1)}$ on the second 
copy of $V_A$, and

\item the chosen $\Ger_{\infty}$ quasi-isomorphism in \eqref{U-Ger-1}. 

\end{itemize}

Due to \cite[Theorem 4.3]{Action}, there exists a degree zero cocycle 
$\cD_{\Cyl}$ in  $\Der'\big( \Cyl(\Ger^{\vee}) \big)$ for which the cocycles 
\begin{equation}
\label{cD-cD'}
\cD : =  \res_1 (\cD_{\Cyl}) \,, \qquad 
\cD' : = \res_2 (\cD_{\Cyl} )
\end{equation}
are both cohomologous to the given cocycle $\cD_2$. 

Precomposing the map $U_{\Cyl}^{V_A}$ with the automorphism 
$\exp(\cD_{\Cyl})$ we get a new  $\Cyl(\Ger^{\vee})$-algebra structure on 
the pair $(V_A, V_A)$ which corresponds to the triple
 
\begin{itemize}

\item the  $\Ger_{\infty}$-structure  $Q^{\exp(\cD)}$ on the first copy of $V_A$, 

\item the $\Ger_{\infty}$-structure  $Q^{\exp ( \CH(\cD_1, \cD')  )}$ on the second 
copy of $V_A$, and

\item a $\Ger_{\infty}$ quasi-isomorphism 

\begin{equation}
\label{wt-U-Ger}
\wt{U}_{\Ger} :  V_A^{Q^{\exp(\cD)}} \leadsto  V_A^{Q^{\exp ( \CH(\cD_1, \cD') )} }\,.
\end{equation}

\end{itemize}

Let us observe that, due to Proposition \ref{prop:morph-Lie-OK} from Appendix 
\ref{app:derivations}, the restriction of $\wt{U}_{\Ger} $ onto $\La^2 \coCom(V_A)$ 
coincides with the restriction of \eqref{U-Ger-1} onto $\La^2 \coCom(V_A)$. Hence, 
\begin{equation}
\label{wt-U-Ger-restr}
\wt{U}_{\Ger} \Big|_{\La^2 \coCom(V_A) } = U^{\cD_1}\,, 
\end{equation}
where $U^{\cD_1}$ is a $\La\Lie_{\infty}$-automorphism of $V_A$ 
corresponding\footnote{Strictly speaking, only the homotopy class of the $\La\Lie_{\infty}$-automorphism  $U^{\cD_1}$
is uniquely determined by $\cD_1$.} to $\cD_1$. 
 
Recall that there exists a $\Ger_{\infty}$ quasi-isomorphism 
\begin{equation}
\label{U-Ger-cD}
U_{\Ger} : V_A  \leadsto  V_A^{Q^{\exp(\cD)}}\,.
\end{equation}
where the source is considered with the standard Gerstenhaber 
structure. Furthermore, since $\cD$ is cohomologous to $\cD_2$, 
the first claim of Proposition \ref{prop:U-cD-OK} implies that 
the restriction of $U_{\Ger}$ onto $\La^2 \coCom(V_A)$ gives us 
a $\La\Lie_{\infty}$-automorphism $U^{\cD}$ of $V_A$ which is 
homotopy equivalent to $U^{\cD_2}$.

Let us also observe that the composition $\wt{U}_{\Ger} \circ U_{\Ger}$ gives 
us a $\Ger_{\infty}$ quasi-isomorphism 
\begin{equation}
\label{desired-comp}
\wt{U}_{\Ger} \circ U_{\Ger} : V_A \leadsto V_A^{Q^{\exp ( \CH(\cD_1, \cD') )} }
\end{equation}
 
Hence, the restriction of  $\wt{U}_{\Ger} \circ U_{\Ger}$ gives us a $\La\Lie_{\infty}$-automorphism 
of $V_A$ corresponding to $\CH(\cD_1, \cD')$. Due to \eqref{wt-U-Ger-restr}, this  $\La\Lie_{\infty}$-automorphism
coincides with 
$$
U^{\cD_1} \circ  U^{\cD}\,.
$$

Since $\cD$ and $\cD'$ are both cohomologous to $\cD_2$, the second claim of 
Proposition \ref{prop:U-cD-OK} follows.  \qed 

\begin{remark}
\label{rem:proof-in-Thomas}
The second claim of Proposition \ref{prop:U-cD-OK} can probably 
be deduced from \cite[Proposition 5.4]{Thomas} and some other 
statements in \cite{Thomas}. However, this would require a digression 
to ``stable setting'' which we avoid in this paper. For this reason, 
we decided to present a complete proof of Proposition \ref{prop:U-cD-OK}
which is independent of any intermediate steps in \cite{Thomas}. 
\end{remark}

\section{Final remarks: connecting Drinfeld associators to the set of homotopy classes  $\pi_0 \big(  V_A \leadsto \Cbu(A) \big)$}
\label{sec:link-to-DrAssoc}

In this section we recall how to construct a $\GRT_1$-equivariant map $\mB$ from the 
set $\DrAssoc_1$ of Drinfeld associators to the set 
$$
\pi_0 \big( \Ger_{\infty} \to \Bra \big)  
$$
of homotopy classes of operad morphisms \eqref{Psi} satisfying 
conditions \eqref{Psi-conditions}. 

Composing $\mB$ with the map $\mT$ \eqref{mT}, we 
get the desired map 
\begin{equation}
\label{mT-mB-here}
\mT \circ \mB :  \DrAssoc_{1} \to \pi_0 \big(  V_A \leadsto \Cbu(A) \big)  
\end{equation}
from the set $\DrAssoc_{1}$ to the set of homotopy classes of 
$\La\Lie_{\infty}$-morphisms from $V_A$ to $\Cbu(A)$ whose linear term 
is the Hochschild-Kostant-Rosenberg embedding.

Theorem \ref{thm:main} will then imply that map \eqref{mT-mB-here}
is $\GRT_1$-equivariant. 

\subsection{The sets $\DrAssoc_{\ka}$ of Drinfeld associators}  
\label{sec:DrAssoc}
In this short subsection, we briefly recall Drinfeld's associators and 
the Grothendieck-Teichmueller group $\GRT_1$\,.
For more details we refer the reader to \cite{AT}, \cite{Bar-Natan}, 
or \cite{Drinfeld}.

Let $m$ be an integer $\ge 2$. We denote by $\mt_m$ the Lie algebra  
generated by symbols $\{ t^{ij} = t^{ji} \}_{1 \le i \neq j \le m}$ subject to the following 
relations: 
\begin{multline}
\label{DK-relations} 
[t^{ij}, t^{ik} + t^{jk}] = 0 \qquad \textrm{for any triple of distinct indices } i,j,k\,,\\
[t^{ij}, t^{kl}] = 0 \qquad \textrm{for any quadruple of distinct indices } i,j,k,l\,.
\end{multline}
The notation $\sA^{\pb}_m$ is reserved for the associative algebra 
(over $\bbK$) of formal power series in noncommutative symbols $\{ t^{ij} = t^{ji} \}_{1 \le i \neq j \le m}$
subject to the same relations \eqref{DK-relations}. Let us recall \cite[Section 4]{Dima-disc} that 
the collection  $\sA^{\pb}  : = \{ \sA^{\pb}_m \}_{m \ge 1}$ with $ \sA^{\pb}_1 : = \bbK$ forms an operad in the category of 
associative $\bbK$-algebras.

Let $\lie(x,y)$ be the degree completion of the free Lie algebra in two symbols $x$ and $y$
and let $\ka$ be any element of $\bbK$. 

The set $\DrAssoc_{\ka}$ consists of elements $\Phi \in \exp \big(\lie(x,y) \big)$
which satisfy the equations 
\begin{equation}
\label{321-inverse}
\Phi(y,x) \Phi(x,y) = 1\,,
\end{equation}

\begin{equation}
\label{pentagon}
\Phi(t^{12}, t^{23} + t^{24}) \, \Phi(t^{13} + t^{23}, t^{34}) = \Phi( t^{23}, t^{34})\, \Phi(t^{12} + 
t^{13}, t^{24} + t^{34})\, \Phi(t^{12}, t^{23})\,, 
\end{equation}

\begin{equation}
\label{hexagon}
e^{\ka (t^{13} + t^{23})/2}  = \Phi(t^{13}, t^{12}) e^{ \ka  t^{13}/2} \Phi(t^{13}, t^{23})^{-1} 
e^{\ka t^{23}/2} \Phi(t^{12}, t^{23})\,,
\end{equation}
and 
\begin{equation}
\label{hexagon-inv}
e^{\ka (t^{12} + t^{13})/2}  = \Phi(t^{23}, t^{13})^{-1} e^{ \ka t^{13}/2} \Phi(t^{12}, t^{13})
e^{\ka t^{12} /2} \Phi(t^{12}, t^{23})^{-1}\,.
\end{equation}

For $\ka \neq 0$, elements $\Phi$ of  $\DrAssoc_{\ka}$ are called 
Drinfeld associators. However, for our purposes, we only need 
the set $\DrAssoc_{1}$ and the set $\DrAssoc_{0}$. 

According to \cite[Section 5]{Drinfeld}, the set 
\begin{equation}
\label{DrAssoc-0}
\DrAssoc_{0}
\end{equation}
forms a prounipotent group and, by \cite[Proposition 5.5]{Drinfeld},
this group acts simply transitively on the set of associators in $\DrAssoc_1$\,. 
Following \cite{Drinfeld}, we denote the group $\DrAssoc_{0}$ by $\GRT_1$.

\subsection{A map $\mB$ from $\DrAssoc_{1}$ to  $\pi_0 \big( \Ger_{\infty} \to \Bra \big)$}

Let us recall \cite{Bar-Natan}, \cite{Dima-disc} that collections of all braid groups can be 
assembled into the operad $\PaB$ in the category of $\bbK$-linear categories. 
Similarly, the collection of algebras $\{\sA^{\pb}_m \}_{m \ge 1}$
can be ``upgraded'' to the operad $\PaCD$ also in the category of $\bbK$-linear 
categories. Every associator $\Phi \in \DrAssoc_{1}$ gives us an isomorphism 
of these operads 
\begin{equation}
\label{I-Phi}
I_{\Phi} :  \PaB \stackrel{\cong}{ \longrightarrow } \PaCD\,.
\end{equation}
The group $\GRT_1$ acts on the operad $\PaCD$ in such a way that, 
for every pair $g \in \GRT_1, ~~\Phi \in  \DrAssoc_{1}$, the diagram 
\begin{equation}
\label{diag-PaB-PaCD}
\begin{tikzpicture}
\matrix (m) [matrix of math nodes, row sep=2.6em,  column sep=3em]
{ \PaB & \PaCD  \\
\PaB & \PaCD \\ };
\path[->,font=\scriptsize]
(m-1-1) edge  node[auto] {$ I_{\Phi} $}  (m-1-2) edge  node[auto] {$ \id $} (m-2-1)
(m-2-1) edge  node[auto] {$ I_{g(\Phi)} $}  (m-2-2) (m-1-2) edge  node[auto] {$g$} (m-2-2);
\end{tikzpicture}
\end{equation}
commutes. 

Applying to $\PaB$ and $\PaCD$ the functor $C_{-\bul}(~, \bbK)$, where 
$C_{\bul}(~, \bbK)$ denotes the Hochschild chain complex 
with coefficients in $\bbK$, we get dg operads
\begin{equation}
\label{C-PaB}
C_{-\bul}(\PaB, \bbK)
\end{equation}
and   
\begin{equation}
\label{C-PaCD}
C_{-\bul}(\PaCD, \bbK)\,.
\end{equation}

By naturality of $C_{-\bul}(~, \bbK)$, diagram \eqref{diag-PaB-PaCD} gives us 
the commutative diagram 
\begin{equation}
\label{diag-Cbd-PaB-PaCD}
\begin{tikzpicture}
\matrix (m) [matrix of math nodes, row sep=2.6em,  column sep=3em]
{ C_{-\bul}(\PaB, \bbK) & C_{-\bul}(\PaCD, \bbK)  \\
C_{-\bul}(\PaB, \bbK) & C_{-\bul}(\PaCD, \bbK), \\ };
\path[->,font=\scriptsize]
(m-1-1) edge  node[auto] {$ I_{\Phi} $}  (m-1-2) edge  node[auto] {$ \id $} (m-2-1)
(m-2-1) edge  node[auto] {$ I_{g(\Phi)} $}  (m-2-2) (m-1-2) edge  node[auto] {$g$} (m-2-2);
\end{tikzpicture}
\end{equation}
where, for simplicity, the maps corresponding to $I_{\Phi}$, $I_{g(\Phi)}$ and $g$ are 
denoted by the same letters, respectively. 
 
Recall that Eq. (5) from \cite{Dima-disc} gives us the canonical quasi-isomorphism 
from the operad $\Ger$ to $C_{-\bul}(\sA^{\pb}, \bbK)$. The latter operad, in turn, 
receives the natural map
$$
C_{-\bul}(\PaCD, \bbK) \to  C_{-\bul}(\sA^{\pb} , \bbK)
$$
from $C_{-\bul}(\PaCD, \bbK)$ which is also known to be a quasi-isomorphism. 

Thus, using the lifting property (see \cite[Corollary 5.8]{notes}) for maps from 
the operad $\Ger_{\infty} = \Cobar(\Ger^{\vee})$, we get the quasi-isomorphism\footnote{By the same lifting 
property (see \cite[Corollary 5.8]{notes}), we know that the homotopy type of the quasi-isomorphism
\eqref{Gerinf-to-PaCD} is uniquely determined by the operad map $\Ger \to C_{-\bul}(\sA^{\pb}, \bbK)$
from  \cite[Eq. (5)]{Dima-disc}.}  
\begin{equation}
\label{Gerinf-to-PaCD}
\Ger_{\infty} \stackrel{\sim}{~\longrightarrow~} C_{-\bul}(\PaCD, \bbK)\,.
\end{equation}

Using this quasi-isomorphism and \cite[Corollary 5.8]{notes}, 
one can construct (see \cite[Section 6.3.1]{Thomas}) a group homomorphism 
\begin{equation}
\label{GRT1-to-exp-mg}
\GRT_1 \to \exp(\mg)\,,
\end{equation}
where the Lie algebra $\mg$ is defined in \eqref{mg}. By \cite[Theorem 1.2]{Thomas}, 
homomorphism \eqref{GRT1-to-exp-mg} is an isomorphism. 

Any specific solution of Deligne's conjecture on the Hochschild complex  (see, for example, 
\cite{BF},  \cite{Swiss}, or \cite{M-Smith}) combined 
with Fiedorowicz's recognition principle \cite{Fied} provides us 
with a sequence of quasi-isomorphisms  
\begin{equation}
\label{PaB-to-Br}
\Bra \,\stackrel{\sim}{\leftarrow}\, \bullet  \, \stackrel{\sim}{\rightarrow}\, \bullet
 \, \stackrel{\sim}{\leftarrow} \bullet ~ \dots ~ \bullet \stackrel{\sim}{\rightarrow} \, C_{-\bul}(\PaB, \bbK)
\end{equation}
which connects the dg operad $\Bra$ to $C_{-\bul}(\PaB, \bbK)$\,.

Hence, every associator $\Phi \in  \DrAssoc_{1}$ gives us a sequence of 
quasi-isomorphisms
\begin{equation}
\label{master-seq}
\Bra  \,\stackrel{\sim}{\leftarrow}\, \bullet  \,  \stackrel{\sim}{\rightarrow}\, \bullet
 \, \stackrel{\sim}{\leftarrow} \bullet ~ \dots ~ \bullet \stackrel{\sim}{\rightarrow} \, C_{-\bul}(\PaB, \bbK)
 \stackrel{I_{\Phi}}{\longrightarrow}
C_{-\bul}(\PaCD, \bbK)  \stackrel{\sim}{~\longleftarrow~}  \Ger_{\infty}
\end{equation}
connecting the dg operads $\Bra$ to $\Ger_{\infty}$. 

Using \cite[Corollary 5.8]{notes} once again, we conclude that  
the sequence of quasi-isomorphisms \eqref{master-seq} determines a unique 
homotopy class of quasi-isomorphisms (of dg operads)
\begin{equation}
\label{Psi-from-Phi}
\Psi: \Ger_{\infty} \to \Bra\,. 
\end{equation}
Thus we get a well defined map 
\begin{equation}
\label{mB}
\mB : \DrAssoc_{1} \to \pi_0 \big( \Ger_{\infty} \to \Bra \big)\,.
\end{equation}

In view of isomorphism \eqref{GRT1-to-exp-mg}, the set 
of homotopy classes $\pi_0 \big( \Ger_{\infty} \to \Bra \big)$ is 
equipped with a natural action of $\GRT_1$. 
Moreover, the commutativity of diagram \eqref{diag-Cbd-PaB-PaCD}
implies that the map $\mB$ is $\GRT_1$-equivariant. 

Thus, combining this observation with Theorem \ref{thm:main} we 
deduce the following corollary:
\begin{cor}
\label{cor:GRT1-equivar}
Let $\pi_0 \big(  V_A \leadsto \Cbu(A) \big)$ be the set of homotopy classes 
of $\La\Lie_{\infty}$ quasi-isomorphisms which extend the Hochschild-Kostant-Rosenberg 
embedding of polyvector fields into Hochschild cochains. If we consider 
$\pi_0 \big(  V_A \leadsto \Cbu(A) \big)$ as a set with the $\GRT_1$-action induced 
by isomorphism \eqref{GRT1-to-exp-mg} then the composition
\begin{equation}
\label{mT-mB}
\mT \circ \mB :  \DrAssoc_{1} \to \pi_0 \big(  V_A \leadsto \Cbu(A) \big) 
\end{equation}
is $\GRT_1$-equivariant. \qed
\end{cor}
\begin{remark}
\label{rem:solutions-Deligne}
Any sequence of quasi-isomorphisms of dg operads \eqref{PaB-to-Br} 
gives us an isomorphism between the objects corresponding to 
$C_{-\bul}(\PaB, \bbK)$ and $\Bra$ in the homotopy category of 
dg operads. However, there is no reason to expect that different solutions 
of the Deligne conjecture give the same isomorphisms from $C_{-\bul}(\PaB, \bbK)$ 
to $\Bra$ in the homotopy category. Hence the resulting 
composition in \eqref{mT-mB} may depend on the choice of 
a specific solution of Deligne's conjecture on the Hochschild complex. 
\end{remark}

\appendix

\section{Filtered $\LLie$-algebras}

Let $L$ be a cochain complex with the differential $\pa$. Recall that 
a  $\LLie$-structure on $L$ is a sequence of degree $1$ multi-brackets 
\begin{equation}
\label{multi-brack}
\{~, ~, \dots, ~ \}_{m}: S^m(L) \to L\,, \qquad m \ge 2 
\end{equation}
satisfying the relations 
\begin{multline}
\label{LLie-relations}
\pa \{v_1, v_2, \dots, v_m\} + 
\sum_{i=1}^m (-1)^{|v_1| + \dots + |v_{i-1}|} 
\{v_1,  \dots, v_{i-1},  \pa v_i, v_{i+1}, \dots, v_{m}\} \\
+ \sum_{k=2}^{m-1} 
\sum_{\si \in \Sh_{k, m-k} } 
(-1)^{\ve(\si; v_1, \dots, v_m)}
\{ \{v_{\si(1)}, \dots, v_{\si(k)} \}, v_{\si(k+1)}, \dots, v_{\si(m)} \} = 0\,,
\end{multline}
where $(-1)^{\ve(\si; v_1, \dots, v_m)}$ is the Koszul sign factor (see eq.  \eqref{ve-si-vvv}). 

We say that a $\LLie$-algebra $L$ is {\it filtered} if it is equipped with  
a complete descending filtration 
\begin{equation}
\label{filtr-L}
L = \cF_{1} L \supset \cF_{2} L \supset \cF_{3} L \supset \dots\,. 
\end{equation}

For such filtered $\LLie$-algebras we may define a Maurer-Cartan element as
a degree zero element $\al$ satisfying the equation 
\begin{equation}
\label{MC-eq}
\pa \al + \sum_{m \ge 2} \frac{1}{m!} \{\al, \al, \dots, \al\}_m   = 0\,.
\end{equation}
Note that this equation makes sense for any degree $0$ element 
$\al$ because $L= \cF_1 L$ and $L$ is complete with respect to filtration \eqref{filtr-L}. 
Let us denote by $\MC(L)$ the set of Maurer-Cartan elements of a filtered $\LLie$-algebra $L$. 

According to\footnote{A version of the Deligne-Getzler-Hinich $\infty$-groupoid for pro-nilpotent 
$\LLie$-algebras is introduced in \cite[Section 4]{Enhanced}.} 
\cite{Getzler}, the set $\MC(L)$ can be upgraded to an $\infty$-groupoid
$\mMC(L)$ (i.e. a simplicial set satisfying the Kan condition). To introduce the $\infty$-groupoid 
$\mMC(L)$,
we  denote by $\Omb(\sfDel_n)$ the dg commutative $\bbK$-algebra of polynomial forms 
\cite[Section 3]{Getzler} on the $n$-th geometric simplex $\sfDel_n$. Next, we declare 
that set of $n$-simplices of $\mMC(L)$ is 
\begin{equation}
\label{MC-n}
\MC \big( L \,\hat{\otimes}\, \Omb(\sfDel_n) \big)\,,
\end{equation}
where $L$ is considered with the topology coming from  
filtration \eqref{filtr-L} and $\Omb(\sfDel_n)$ is considered with the 
discrete topology. The structure of the simplicial set is induced from the 
structure of a simplicial set on the sequence $\{ \Omb(\sfDel_n)\}_{n \ge 0}$\,.

For example, $0$-cells of $\mMC(L)$ are precisely Maurer-Cartan elements of $L$
and $1$-cells are sums 
\begin{equation}
\label{1-cell-al}
\al' + d t \, \al'' \,, \qquad \al' \in L^0 \hotimes \bbK[t]\,, \qquad 
\al'' \in  L^{-1} \hotimes \bbK[t]
\end{equation}
satisfying the pair of equations
\begin{equation}
\label{MC-al-pr}
\pa \al' + \sum_{m \ge 2} \frac{1}{m!} \{\al', \al', \dots, \al'\}_m   = 0\,,
\end{equation}
\begin{equation}
\label{MC-al-prpr}
\frac{d}{d t} \al'  = \pa \al''   + \sum_{m \ge 1} \frac{1}{m!} \{\al', \al', \dots, \al',  \al'' \}_{m+1}\,.
\end{equation}

Thus, two $0$-cells $\al_0$, $\al_1$ of  $\mMC(L)$ (i.e. Maurer-Cartan elements of $L$) 
are isomorphic if there exists an element \eqref{1-cell-al} satisfying 
\eqref{MC-al-pr} and \eqref{MC-al-prpr} and such that 
\begin{equation}
\label{al-0-al-1}
\al_0 =  \al'  \Big|_{t = 0} \qquad 
\textrm{and} \qquad
\al_1 =  \al'  \Big|_{t = 1}\,.
\end{equation}
We say that a $1$-cell  \eqref{1-cell-al} connects $\al_0$ and $\al_1$. 

\subsection{A lemma on adjusting Maurer-Cartan elements}
\label{app:the-lemma}

Let $\al$ be a Maurer-Cartan element of a filtered $\LLie$-algebra
and $\xi$ be a degree $-1$ element in $\cF_n L$ for some 
integer $n \ge 1$.   

Let us consider the following sequence $\{\al'_k\}_{k \ge 0}$ of degree zero elements 
in $L \hotimes \bbK[t]$
\begin{equation}
\label{al-pr-seq}
\al'_0 : = \al\,, \qquad \al'_{k+1}(t) : = \al + 
\int_{0}^t d t_1 \Big( \pa \xi + \sum_{m \ge 1} \frac{1}{m!} 
\{\al'_k(t_1), \dots, \al'_k(t_1), \xi \}_{m+1} \Big)\,.
\end{equation}

Since $L$ is complete with respect to filtration \eqref{filtr-L}, the sequence 
$\{\al'_k\}_{k \ge 0}$ convergences to a (degree $0$) element $\al' \in  L \hotimes \bbK[t]$ which 
satisfies the integral equation
\begin{equation}
\label{al-pr-xi-int}
\al'(t)  = \al + 
\int_{0}^t d t_1 \Big( \pa \xi + \sum_{m \ge 1} \frac{1}{m!} 
\{\al'(t_1), \dots, \al'(t_1), \xi \}_{m+1} \Big)\,. 
\end{equation}

We claim that 
\begin{lemma}
\label{lem:isom-MC}
If, as above, $\xi$ is a degree $-1$ element in $\cF_n L$ and 
$\al'$ is an element of $L \hotimes \bbK[t]$ obtained by  
recursive procedure \eqref{al-pr-seq} then the sum 
\begin{equation}
\label{the-1-cell}
\al' + dt \, \xi
\end{equation}
is a $1$-cell of $\mMC(L)$ which connects $\al$ to 
another Maurer-Cartan element $\wt{\al}$ of $L$ 
such that
\begin{equation}
\label{mod-cF-n}
\al' - \al \in  \cF_{n} L \hotimes \bbK[t]\,,
\end{equation}
and 
\begin{equation}
\label{wtal-al-diff-xi}
\wt{\al} - \al - \pa \xi \in \cF_{n+1} L\,.
\end{equation}
If the element $\xi$ satisfies the additional condition 
\begin{equation}
\label{diff-xi-cFn1}
\pa \xi \in \cF_{n+1} L
\end{equation}
then 
\begin{equation}
\label{mod-cF-n1}
\al' - \al \in  \cF_{n+1} L \hotimes \bbK[t]\,,
\end{equation}
and 
\begin{equation}
\label{wtal-al-diff-xi-better}
\wt{\al} - \al - \pa \xi - \{\al, \xi\} \in \cF_{n+2} L\,.
\end{equation}
\end{lemma}
\begin{proof} Equation \eqref{al-pr-xi-int} implies that 
$\al'$ satisfies the differential equation
\begin{equation}
\label{al-pr-xi}
\frac{d}{dt} \al' =  \pa \xi + \sum_{m \ge 1} \frac{1}{m!} 
\{\al', \dots, \al', \xi \}_{m+1} 
\end{equation}
with the initial condition 
\begin{equation}
\label{al-pr-initial}
\al' \Big|_{t = 0} = \al\,. 
\end{equation}

Let us denote by $\Xi$ the following degree $1$ element of $L \hotimes \bbK[t]$
\begin{equation}
\label{Xi}
\Xi : = \pa \al'  + \sum_{m \ge 2} \frac{1}{m!} \{\al', \al', \dots, \al'\}_m\,.  
\end{equation}

A direct computation shows that $\Xi$ satisfies the following differential 
equation
\begin{equation}
\label{diff-eq-Xi}
\frac{d}{d t} \Xi = - \sum_{m \ge 0} \frac{1}{m!} \{\al', \dots, \al', \Xi, \xi\}_{m+2}\,.
\end{equation}
 
Furthermore, since $\al$ is a Maurer-Cartan element of $L$, the element $\Xi$ satisfies 
the condition 
$$
\Xi \Big |_{t=0} = 0
$$
and hence $\Xi$ satisfies the integral equation
\begin{equation}
\label{int-eq-Xi}
\Xi (t) = - \int_0^t d t_1 \Big( 
 \sum_{m \ge 0} \frac{1}{m!} \{\al'(t_1), \dots, \al'(t_1), \Xi(t_1), \xi\}_{m+2}
\Big)\,.
\end{equation}

Equation \eqref{int-eq-Xi} implies that 
$$
\Xi \in \bigcap_{n \ge 1} \cF_n L \hotimes \bbK[t]\,.
$$

Therefore $\Xi = 0$ and hence the limiting element $\al'$
of sequence \eqref{al-pr-seq} is a Maurer-Cartan element of $L \hotimes \bbK[t]$\,.

Combining this observation with differential equation \eqref{al-pr-xi}, we conclude 
that the element $\al' + dt \, \xi \in L \hotimes \Omb(\sfDel_1)$ is indeed a $1$-cell in 
$\mMC(L)$ which connects the Maurer-Cartan element $\al$ to the  Maurer-Cartan element 
\begin{equation}
\label{wt-al}
\wt{\al} : = \al + \int_{0}^1 d t \Big( \pa \xi + \sum_{m \ge 1} \frac{1}{m!} 
\{\al'(t), \dots, \al'(t), \xi \}_{m+1} \Big)\,. 
\end{equation}

Since $\xi \in \cF_n L$ and $L = \cF_1 L$, equation \eqref{al-pr-xi-int} implies that 
$$
\al' - \al  \in  \cF_{n}  L \hotimes \bbK[t]
$$
and equation \eqref{wt-al} implies that 
$$
\wt{\al} - \al - \pa \xi \in \cF_{n+1} L\,.
$$
Thus, the first part of Lemma \ref{lem:isom-MC} is proved.

If $\xi \in  \cF_{n} L$ and $\pa \xi \in \cF_{n+1} L$ then, again, 
it is clear from \eqref{al-pr-xi-int} that inclusion \eqref{mod-cF-n1}
holds.

Finally, using inclusion \eqref{mod-cF-n1} and 
equation \eqref{wt-al}, it is easy to see that 
$$
\wt{\al} - \al - \pa \xi - \{\al, \xi\} \in \cF_{n+2} L\,.
$$

Lemma \ref{lem:isom-MC} is proved. 
\end{proof}

\subsection{Convolution $\LLie$-algebra, $\infty$-morphisms and their homotopies}
\label{app:conv}

Let $\cC$ be a coaugmented cooperad (in the category of graded vector spaces)
satisfying the additional condition 
\begin{equation}
\label{cC-reduced}
\cC(0) = \bfzero 
\end{equation}
and $V$ be a cochain complex. (In this paper, $\cC$ is usually the cooperad $\Ger^{\vee}$.)  

Following \cite{DHR}, we say that $V$ is a homotopy algebra of type $\cC$
if $V$ carries $\Cobar(\cC)$-algebra structure or equivalently the $\cC$-coalgebra 
$$
\cC(V)
$$
has a degree $1$ coderivation $Q$ satisfying 
$$
Q \Big|_V = 0
$$
and the  Maurer-Cartan equation 
$$
[d_V, Q] + \frac{1}{2}[Q,Q] = 0
$$
where $d_V$ is the differential on $\cC(V)$ induced from the one on $V$. 

For two homotopy algebras $(V, Q_V)$ and $(W, Q_W)$ of type $\cC$, we consider 
the graded vector space
\begin{equation}
\label{Map-V-W}
\Hom(\cC(V), W)
\end{equation}
with the differential $\pa$
\begin{equation}
\label{diff-Map-VW}
\pa(f) : = d_W \circ f - (-1)^{|f|} f \circ (d_V + Q_V) 
\end{equation}
and the multi-brackets (of degree $1$)
$$
\{~,~, \dots, ~ \}_m : S^m \big( \Hom(\cC(V), W) \big) \to \Hom(\cC(V), W)\,, \qquad m \ge 2
$$  
\begin{equation}
\label{m-bracket}
\{f_1, \dots, f_m\} (X) = 
p_W \circ Q_W \big( 1 \otimes f_{1}\otimes \dots \otimes f_{m}  
(\D_m (X))  \big)\,, 
\end{equation}
where $\D_m$ is the $m$-th component of the comultiplication 
$$
\D_m : \cC(V) \to \Big( \cC(m) \otimes \cC(V)^{\otimes\, m} \Big)^{S_m}
$$
and $p_W$ is the canonical projection 
$$
p_W : \cC(W) \to W\,.
$$
 
According to \cite{DHR} or \cite[Section 1.3]{3Tales},  equation \eqref{m-bracket} define 
a $\La^{-1}\Lie_{\infty}$-structure on the cochain complex $\Hom(\cC(V), W)$ with the 
differential $\pa$ \eqref{diff-Map-VW}. The $\LLie$-algebra  
\begin{equation}
\label{Hom-cC-V-W}
\Hom(\cC(V), W)
\end{equation}
is called the {\it convolution  $\LLie$-algebra} of the pair $V, W$. 

The convolution $\LLie$-algebra  $\Hom(\cC(V), W)$ carries the obvious 
descending filtration ``by arity''
\begin{equation} 
\label{filtr-Map-VW}
\cF_{n} \Hom(\cC(V), W)  ~ = ~ \{ f  \in \Hom(\cC(V), W) ~\vert ~ 
f \big|_{\cC(m) \otimes_{S_m} V^{\otimes \, m}} =0
~~ \forall ~ m < n  \}.
\end{equation}

$\Hom(\cC(V), W)$ is obviously complete with respect to this 
filtration and 
\begin{equation}
\label{Map-VW-is-cF_1}
\Hom(\cC(V), W) = \cF_{1} \Hom(\cC(V), W)
\end{equation}
due to condition \eqref{cC-reduced}. In other words, under our assumption 
on the cooperad $\cC$, the convolution $\LLie$-algebra  $\Hom(\cC(V), W)$  
is pronilpotent. 

According to \cite[Proposition 3]{3Tales}, $\infty$-morphisms from $V$ to $W$ 
are in bijection with Maurer-Cartan elements of  $\Hom(\cC(V), W)$ i.e. $0$-cells of 
the Deligne-Getzler-Hinich $\infty$-groupoid corresponding to  $\Hom(\cC(V), W)$.  
Furthermore, due to \cite[Corollary 2]{3Tales}, two $\infty$-morphisms
from $V$ to $W$ are homotopic if and only if the corresponding  Maurer-Cartan elements
are isomorphic $0$-cells in the Deligne-Getzler-Hinich $\infty$-groupoid 
of $\Hom(\cC(V), W)$. 

\section{Tamarkin's rigidity}
\label{app:rigid}

Let $V_A$ denote the Gerstenhaber algebra of polyvector fields 
on the graded affine space corresponding to 
$A= \bbK[x^1, x^2, \dots, x^d]$ with 
$$
| x^i | = t_i\,.
$$

As the graded commutative algebra over $\bbK$, 
$V_A$ is freely generated by variables 
$$
x^1, x^2, \dots, x^d, \te_1, \te_2, \dots, \te_d,
$$ 
where $\te_i$ carries degree $1 - t_i$\,. 
\begin{equation}
\label{V-A}
V_A = \bbK [x^1, x^2, \dots, x^d, \te_1, \te_2, \dots, \te_d]\,.
\end{equation}
Let us denote by $\mu_{\wedge}$ and $\mu_{\{~,~\}}$ the vectors 
in $\End_{V_A}(2)$ corresponding to the multiplication 
and the Schouten bracket $\{~,~ \}$ on $V_A$, respectively. 

The composition of the canonical quasi-isomorphism 
$$
\Cobar(\Ger^{\vee}) \to \Ger
$$
and the map $\Ger \to \End_{V_A}$ corresponds to the 
following  Maurer-Cartan element 
\begin{equation}
\label{al}
\al : = \mu_{\wedge} \otimes \{b_1, b_2\} + \mu_{\{~,~\}} \otimes b_1 b_2
\end{equation}
in the graded Lie algebra
\begin{equation}
\label{Conv-End-VA}
\Conv^{\oplus} (\Ger^{\vee}, \End_{V_A}) : = 
\bigoplus_{n \ge 1} \Hom_{S_n} \left(\Ger^{\vee}(n),  \End_{V_A}(n) \right) 
\end{equation}
for which we frequently use the obvious 
identification\footnote{Recall that the cooperad $\Ger^{\vee}$ is the linear 
dual of the operad $\La^{-2} \Ger$\,.} 
\begin{equation}
\label{Conv-is-otimes}
\Conv^{\oplus} (\Ger^{\vee}, \End_{V_A}) \cong
\bigoplus_{n \ge 1} 
\left( \End_{V_A}(n) \otimes \La^{-2} \Ger(n) \right)^{S_n}\,. 
\end{equation}

In this section, we consider $\Conv^{\oplus} (\Ger^{\vee}, \End_{V_A})$ as 
the cochain complex with the following differential  
\begin{equation}
\label{diff-Conv}
\pa : = [\al, ~]\,.
\end{equation}

We observe that  $\Conv^{\oplus} (\Ger^{\vee}, \End_{V_A})$ carries the 
natural descending filtration ``by arity'': 
$$
\Conv^{\oplus} (\Ger^{\vee}, \End_{V_A}) = \cF_0 \Conv^{\oplus} (\Ger^{\vee}, \End_{V_A}) 
\supset  \cF_1 \Conv^{\oplus} (\Ger^{\vee}, \End_{V_A}) \supset \dots
$$
\begin{equation}
\label{filtr-arity}
\cF_m \Conv^{\oplus} (\Ger^{\vee}, \End_{V_A}) : =  
\bigoplus_{n \ge m+1} 
\left( \End_{V_A}(n) \otimes \La^{-2} \Ger(n) \right)^{S_n} \,.
\end{equation}

More precisely, 
\begin{equation}
\label{diff-n-to-n1}
\pa \left( \End_{V_A}(n) \otimes \La^{-2} \Ger(n) \right)^{S_n} ~ \subset ~ 
\left( \End_{V_A}(n+1) \otimes \La^{-2} \Ger(n+1) \right)^{S_{n+1}}\,.
\end{equation}
In particular, every cocycle  $X \in \Conv^{\oplus} (\Ger^{\vee}, \End_{V_A})$ is 
a finite sum
\begin{equation}
\label{X-sum}
X = \sum_{n \ge 1} X_n\,, \qquad X_n \in  \left( \End_{V_A}(n) \otimes \La^{-2} \Ger(n) \right)^{S_n}
\end{equation}
where each individual term $X_n$ is a cocycle.  
 
In this paper, we need the following version of Tamarkin's rigidity
\begin{thm}
\label{thm:rigidity}
If $n$ is an integer $\ge 2$ then for every cocycle 
$$
X ~\in~  \left( \End_{V_A}(n) \otimes \La^{-2} \Ger(n) \right)^{S_n} ~\subset~  \Conv^{\oplus} (\Ger^{\vee}, \End_{V_A})
$$
there exists a cochain $Y \in \left( \End_{V_A}(n-1) \otimes \La^{-2} \Ger(n-1) \right)^{S_{n-1}}$ such that
$$
X= \pa Y\,.
$$ 
\end{thm}
\begin{remark}
\label{rem:versus-stable}
Note that the above statement is different from Tamarkin's rigidity 
in the ``stable setting'' \cite[Section 12]{notes}. 
According to \cite[Corollary 12.2]{notes}, one may think that the vector 
$$
 \mu_{\{~,~\}} \otimes b_1 b_2
$$
is a non-trivial cocycle in \eqref{Conv-End-VA}. In fact, 
$$
 \mu_{\{~,~\}} \otimes b_1 b_2 = [\al, P \otimes b_1 ]\,,
$$
where $P$ is the following version of the ``Euler derivation'' of $V_A$. 
$$
P(v) : = \sum_{i =1}^d \te_i \frac{\pa}{\pa \te_i}\,.
$$
\end{remark}
\begin{proof}[~of Theorem \ref{thm:rigidity}]
Theorem \ref{thm:rigidity} is only a slight generalization of the statement proved 
in Section 5.4 of \cite{Hinich} and, in the proof given here, we pretty much 
follow the same line of arguments as in \cite[Section 5.4]{Hinich}.

First, we introduce an additional set of auxiliary variables
\begin{equation}
\label{new-variables}
\ck{x}_1, \check{x}_2, \dots, \check{x}_d,~ \check{\te}^1, \check{\te}^2, \dots,
\check{\te}^d      
\end{equation}
of degrees 
$$
| \ck{x}_i | = 2- t_i \,, \qquad |\ck{\te}^i | = t_i + 1\,.
$$

Second, we consider the de Rham complex of $V_A$: 
\begin{equation}
\label{de-Rham-VA}
\Om^{\bul}_{\bbK} V_A : = V_A[\ck{x}_1, \ck{x}_2, \dots, \ck{x}_d, \ck{\te}_1,  \ck{\te}_2, \dots,  \ck{\te}_d]
\end{equation}
with the differential 
\begin{equation}
\label{de-Rham}
D = \sum_{i=1}^d  \ck{x}_i \frac{\pa}{\pa \te_i} +   \sum_{i=1}^d  \ck{\te}^i \frac{\pa}{\pa x^i}
\end{equation}
and equip it with the following descending filtration: 
\begin{multline}
\label{filtr-de-Rham}
\cF_m  \Om^{\bul}_{\bbK} V_A   : = 
\Big\{  P \in V_A[\ck{x}_1, \ck{x}_2, \dots, \ck{x}_d, \ck{\te}_1,  \ck{\te}_2, \dots,  \ck{\te}_d] 
\\ \big|~~
\textrm{the total degree of }P\textrm{ in }\ck{x}_1,  \dots, \ck{x}_d, \ck{\te}_1,  \dots,  \ck{\te}_d
\textrm{ is } \ge m+1
\Big\}\,.
\end{multline}

Next, we observe that every homogeneous vector\footnote{Summation over repeated indices is assumed.} 
$$
P = P^{i_1 i_2 \dots i_{k}}_{j_1 j_2 \dots j_q} \,  \ck{x}_{i_1}  \dots \ck{x}_{i_k}
 \ck{\te}^{j_1}  \dots \ck{\te}^{j_q} \in V_A[\ck{x}_1, \ck{x}_2, \dots, \ck{x}_d, \ck{\te}_1,  \ck{\te}_2, \dots,  \ck{\te}_d]
$$
defines an element $P^{\End} \in  \End_{V_A}(k+q)$: 
\begin{multline}
\label{P-End}
P^{\End} (v_1, v_2, \dots, v_{k+q}) : = \\
 \sum_{\si \in S_{k+q} } \pm P^{i_1 i_2 \dots i_{k}}_{j_1 j_2 \dots j_q}
\pa_{ x^{i_1} } v_{\si(1)}\, \pa_{ x^{i_2} } v_{\si(2)}\,  \dots  \pa_{ x^{i_k} } v_{\si(k)}\, \\
\pa_{ \te_{j_1} } v_{\si(k+1)} \,    \pa_{ \te_{j_2} } v_{\si(k+2)} \, \dots  
\pa_{ \te_{j_q} } v_{\si(k+q)}\,,
\end{multline}
where the sign factors $\pm$ are determined by the usual Koszul rule. 

Finally, we claim that the formula
\begin{equation}
\label{VH-dfn}
\VH(P) : = P^{\End} \otimes b_1 b_2 \dots b_{k+q}
\end{equation}
defines a degree zero injective map 
\begin{equation}
\label{VH}
\VH : \bs^{-2}\, \cF_0 \Om^{\bul}_{\bbK} V_A ~\to~
\Conv^{\oplus} (\Ger^{\vee}, \End_{V_A})
\end{equation}
which is compatible with filtrations \eqref{filtr-arity} and \eqref{filtr-de-Rham}. 

A direct computation shows that $\VH$ intertwines differentials
\eqref{diff-Conv} and \eqref{de-Rham}. 

Let $m$ be an integer and  
\begin{equation}
\label{cG-m}
\cG^m \Conv^{\oplus} (\Ger^{\vee}, \End_{V_A}) 
\end{equation}
be the subspace of  $\Conv^{\oplus} (\Ger^{\vee}, \End_{V_A})$ of sums 
\begin{equation}
\label{sums}
\sum_i M_i \otimes q_i ~ \in ~ \bigoplus_{n \ge 1} 
\big( \End_{V_A}(n) \otimes \La^{-2}\Ger(n) \big)^{S_n}  
\end{equation}
satisfying the condition 
\begin{equation}
\label{Lie-brack-deg-le-m}
\textrm{the number of Lie brackets in } q_i ~ - ~ |\, M_i \otimes q_i\, |  ~\le~  m\,.
\end{equation}

It is easy to see that the sequence of subspaces \eqref{cG-m}
$$
\dots \subset
\cG^{-1}  \Conv^{\oplus} (\Ger^{\vee}, \End_{V_A}) \subset 
\cG^{0}  \Conv^{\oplus} (\Ger^{\vee}, \End_{V_A}) \subset
\cG^{1}  \Conv^{\oplus} (\Ger^{\vee}, \End_{V_A}) \subset \dots
$$
form an ascending filtration on the cochain complex 
$\Conv^{\oplus} (\Ger^{\vee}, \End_{V_A})$ and the associated 
graded cochain complex 
\begin{equation}
\label{Gr-cG}
\Gr_{\cG} \Conv^{\oplus} (\Ger^{\vee}, \End_{V_A}) 
\end{equation}
is isomorphic to 
$$
\bigoplus_{n \ge 1} \big( \End_{V_A}(n) \otimes \La^{-2}\Ger(n) \big)^{S_n} 
$$ 
with the differential 
\begin{equation}
\label{diff-Gr}
\pa^{\Gr} = [\, \mu_{\wedge} \otimes \{b_1, b_2\}, ~~ ]\,, 
\end{equation}
where $\mu_{\wedge}$ is the vector in  $\End_{V_A}(2)$ which 
corresponds to the multiplication on $V_A$. 
 
Let us observe that \eqref{Gr-cG} is naturally a $V_A$-module
(where $V_A$ is viewed as the graded commutative algebra), 
differential \eqref{diff-Gr} is $V_A$-linear, and since 
$$
\Ger^{\vee}(V_A) = \La^2 \coCom (\La\coLie(V_A))\,,
$$
cochain complex \eqref{Gr-cG} is isomorphic to 
\begin{equation}
\label{Gr-cG-simpler}
\Hom_{V_A}\big( \bs^2 \und{S}_{V_A} (\bsi V_A \otimes_{\bbK} \coLie (\bsi V_A) ), V_A \big)
\end{equation}
with the differential coming from the one on the Harrison 
homological\footnote{The cochain complex in \eqref{Harr-VA} is obtained from  
the conventional Harrison homological complex from \cite[Section 4.2.10]{Loday} 
by reversing the grading.} 
complex \cite[Section 4.2.10]{Loday}
\begin{equation}
\label{Harr-VA}
V_A \otimes_{\bbK} \coLie (\bsi V_A)
\end{equation}
of the graded commutative algebra $V_A$ with coefficients in $V_A$. 

Since $V_A$ is freely generated by elements 
$x^1, \dots, x^d, \te_1, \dots, \te_d$,  Theorem 3.5.6 and 
Proposition 4.2.11 from \cite{Loday} imply that 
the embedding 
\begin{equation}
\label{I-Harr}
I_{\Harr} ~ : ~\bigoplus_{i=1}^d V_A  e^i ~ \oplus ~
 \bigoplus_{i=1}^d V_A  f_i 
~\to~ V_A \otimes \coLie(\bsi V_A)  
\end{equation}
$$
I_{\Harr}(e^i) : = 1 \otimes \bsi x^i\,, \qquad 
I_{\Harr}(f_i) : = 1 \otimes \bsi \te_i
$$
from the free $V_A$-module 
\begin{equation}
\label{free-dude}
\bigoplus_{i=1}^d V_A  e^i ~ \oplus ~
 \bigoplus_{i=1}^d V_A  f_i\,, \qquad  |e^i| := t_i-1, \qquad  |f_i| : = -t_i
\end{equation}
is a quasi-isomorphism of cochain complexes of $V_A$-modules
from \eqref{free-dude} with the zero differential to \eqref{Harr-VA} with 
the Harrison differential.

Since \eqref{I-Harr} is a quasi-isomorphism of cochain complexes of 
free $V_A$-modules, it induces a quasi-isomorphism of cochain complexes 
of (free) $V_A$-modules:
\begin{equation}
\label{I-Harr-symm}
\bs^2 V_A [\bsi e^1, \dots, \bsi e^d, \bsi f_1, \dots, \bsi f_d] \to  
 \bs^2 S_{V_A} (\bsi V_A \otimes_{\bbK} \coLie (\bsi V_A) )\,,
\end{equation}
where the source carries the zero differential.  

Therefore, map \eqref{VH} induces a quasi-isomorphism 
of cochain complexes  
$$
\bs^{-2}\, \cF_0 \Om^{\bul}_{\bbK} V_A ~\to~
\Gr_{\cG} \Conv^{\oplus} (\Ger^{\vee}, \End_{V_A})\,,
$$ 
where the source is considered with the zero differential.

Thus, by Lemma A.3 from \cite{notes}, map \eqref{VH} is 
a quasi-isomorphism of cochain complexes. 

Let $n \ge 2$ and
\begin{equation}
\label{X-start}
X ~\in~  \left( \End_{V_A}(n) \otimes \La^{-2} \Ger(n) \right)^{S_n} ~\subset~  \Conv^{\oplus} (\Ger^{\vee}, \End_{V_A})
\end{equation}
be a cocycle. 

Since  \eqref{VH} is a quasi-isomorphism of cochain complexes, there
exists a cocycle
\begin{equation}
\label{wt-X}
\wt{X} \in \bs^{-2}\, \cF_0 \Om^{\bul}_{\bbK} V_A
\end{equation}
such that $X$ is cohomologous to $\VH(\wt{X})$\,.

Let us observe that de Rham differential $D$ \eqref{de-Rham} satisfies the property
$$
D  \big( \cF_0 \Om^{\bul}_{\bbK} V_A  \big) \subset  \cF_1 \Om^{\bul}_{\bbK} V_A\,.
$$
Hence, since $\VH$ is injective, we conclude that  
\begin{equation}
\label{wt-X-in-cF1}
\wt{X} \in \bs^{-2}\, \cF_1 \Om^{\bul}_{\bbK} V_A\,.
\end{equation}

It is obvious that every cocycle in $\cF_1 \Om^{\bul}_{\bbK} V_A$ is 
exact in $\cF_0 \Om^{\bul}_{\bbK} V_A$. Therefore $\wt{X}$ is exact and 
so is cocycle \eqref{X-start}. 

Combining this statement with property \eqref{diff-n-to-n1} we easily 
deduce Theorem \ref{thm:rigidity}. 
\end{proof}

\subsection{The standard Gerstenhaber structure on $V_A$ is ``rigid''}
\label{app:VA-to-VA}

The first consequence of Theorem \ref{thm:rigidity} is the 
following corollary:
\begin{cor}
\label{cor:VA-VAQ}
Let $V_A$ be, as above, the algebra of polyvector fields 
on a graded affine space and $Q$ be a $\Ger_{\infty}$-structure 
on $V_A$ whose binary operations are the Schouten bracket and 
the usual multiplication. Then the identity map $\id : V_A \to V_A$
can be extended to a $\Ger_{\infty}$ morphism 
\begin{equation}
\label{cor-Ger-infty}
U_{\corr}  :  V_A \leadsto V^Q_A  
\end{equation}
from $V_A$ with the standard Gerstenhaber structure to 
$V_A$ with the $\Ger_{\infty}$-structure $Q$\,.
\end{cor}
\begin{proof}
To prove this statement, we consider the graded space 
\begin{equation}
\label{Hom-Ger-VA-VA}
\Hom(\Ger^{\vee}(V_A), V_A)
\end{equation}
with two different algebraic structures. 
First, \eqref{Hom-Ger-VA-VA} is identified with the convolution 
Lie algebra\footnote{In our case, Lie algebra \eqref{Conv-Ger-EndVA} carries 
the zero differential.} 
\begin{equation}
\label{Conv-Ger-EndVA}
\Conv(\Ger^{\vee}, \End_{V_A})
\end{equation}
with the Lie bracket $[~,~]$ defined in terms of the binary 
(degree zero) operation $\bullet$ from \cite[Section 4, Eq. (4.2)]{notes}. 

To introduce the second algebraic structure on  \eqref{Hom-Ger-VA-VA}, 
we recall that a $\Ger_{\infty}$-structure on $V_A$ is 
precisely a degree $1$ element
\begin{equation}
\label{Q}
Q = Q_2 + \sum_{n \ge 3} Q_n \qquad Q_n  \in \Hom_{S_n} (\Ger^{\vee}(n) \otimes V_A^{\otimes \, n}, V_A) 
\end{equation} 
in \eqref{Conv-Ger-EndVA} satisfying the  Maurer-Cartan equation
\begin{equation}
\label{MC-Q}
[Q, Q] = 0
\end{equation}
and the above condition on the binary operations is equivalent to 
the requirement 
\begin{equation}
\label{Q2-is-alpha}
Q_2 = \al\,,
\end{equation}
where $\al$ is Maurer-Cartan element \eqref{al} of \eqref{Conv-Ger-EndVA}. 

Given such a $\Ger_{\infty}$-structure $Q$ on $V_A$, we 
get the convolution $\La^{-1}\Lie_{\infty}$-algebra 
\begin{equation}
\label{map-VA-VAQ}
\Hom(\Ger^{\vee}(V_A), V_A^Q)
\end{equation}
corresponding to the pair $(V_A, V_A^Q)$, where the first entry $V_A$ 
is considered with the standard Gerstenhaber structure and the second 
entry is considered with the above $\Ger_{\infty}$-structure $Q$.

As a graded vector space, $\La^{-1}\Lie_{\infty}$-algebra 
\eqref{map-VA-VAQ} coincides with \eqref{Hom-Ger-VA-VA}.
However, it carries a non-zero differential $d_{\al}$ given by the formula
\begin{equation}
\label{diff-al}
d_{\al} (P) = - (-1)^{|P|} P \bullet \al\,, 
\end{equation}
and the corresponding (degree $1$) brackets 
$$
\{~,~, \dots, ~\}_k : S^k \big( \Hom(\Ger^{\vee}(V_A), V_A^Q) \big) \to 
\Hom(\Ger^{\vee}(V_A), V_A^Q)
$$
are defined by general formula \eqref{m-bracket}
in terms of the $\Ger^{\vee}$-coalgebra structure on 
$\Ger^{\vee}(V_A)$ and the $\Ger_{\infty}$-structure $Q$ on $V_A$. 

Let us recall \cite{DHR}, \cite{3Tales} that $\Ger_{\infty}$-morphisms 
from $V_A$ to $V_A^Q$ are in bijection with 
Maurer-Cartan elements\footnote{Recall that  Maurer-Cartan elements of a $\La^{-1}\Lie_{\infty}$-algebra
have degree $0$.}
\begin{equation}
\label{beta}
\beta = \sum_{n \ge 1} \beta_n \,, \qquad 
\beta_n \in \Hom_{S_n} (\Ger^{\vee}(n) \otimes V_A^{\otimes \, n}, V_A) 
\end{equation}
of $\La^{-1}\Lie_{\infty}$-algebra \eqref{map-VA-VAQ} such that 
$\beta_1$ corresponds to the linear term of the corresponding  
$\Ger_{\infty}$-morphism. 

Thus our goal is to prove that, for every Maurer-Cartan element $Q$ \eqref{Q}
of Lie algebra \eqref{Conv-Ger-EndVA} satisfying condition \eqref{Q2-is-alpha},
there exists a  Maurer-Cartan element $\beta$ (see \eqref{beta}) of $\La^{-1}\Lie_{\infty}$-algebra 
\eqref{map-VA-VAQ} such that 
\begin{equation}
\label{beta1-id}
\beta_1 = \id : V_A \to V_A\,.
\end{equation}

Condition \eqref{Q2-is-alpha} implies that the element
$$
\beta^{(1)} : = \id \in \Hom(\Ger^{\vee}(V_A), V_A^Q)
$$
satisfies the equation (in the $\La^{-1}\Lie_{\infty}$-algebra $\Hom(\Ger^{\vee}(V_A), V_A^Q)$)
\begin{equation}
\label{MC-beta-1} 
\Big(\, d_{\al} (\beta^{(1)})  +  \sum_{k \ge 2} 
\frac{1}{k!}  \{  \beta^{(1)}, \dots, \beta^{(1)}  \}_k \, \Big) (X) = 0 
\end{equation}
for every $X \in (\Ger^{\vee}(m) \otimes V_A^{\otimes\, m})_{S_m}$ with $m \le 2$\,.

Let us assume that we constructed (by induction) a degree zero element 
\begin{equation}
\label{beta-1n}
\beta^{(n-1)} = \id + \beta_2 + \beta_3 + \dots + \beta_{n-1} \,, \qquad 
\beta_j \in \Hom_{S_j} (\Ger^{\vee}(j) \otimes V_A^{\otimes \, j}, V_A) 
\end{equation}
such that 
\begin{equation}
\label{MC-beta-1n} 
\Big(d_{\al} (\beta^{(n-1)})  +  \sum_{k \ge 2} \frac{1}{k!}  
\{ \beta^{(n-1)}, \dots, \beta^{(n-1)} \}_{k} \Big) (X) = 0 
\end{equation}
for every $X \in (\Ger^{\vee}(m) \otimes V_A^{\otimes\, m})_{S_m}$ with $m \le n$\,.

We will try to find an element 
\begin{equation}
\label{beta-n}
\beta_n \in \Hom_{S_n} (\Ger^{\vee}(n) \otimes V_A^{\otimes \, n}, V_A)
\end{equation}
such that the sum
\begin{equation}
\label{beta-n-whole}
\beta^{(n)} : = \id + \beta_2 + \beta_3 + \dots + \beta_{n-1} + \beta_n
\end{equation}
satisfies the equation 
\begin{equation}
\label{MC-beta-n} 
\Big(d_{\al} (\beta^{(n)})  +  \sum_{k \ge 2} \frac{1}{k!}  
\{ \beta^{(n)}, \dots, \beta^{(n)} \}_{k} \Big) (X) = 0 
\end{equation}
for every $X \in (\Ger^{\vee}(m) \otimes V_A^{\otimes\, m})_{S_m}$ with $m \le n +1$\,. 

Since $\beta_n \in \Hom_{S_n} (\Ger^{\vee}(n) \otimes V_A^{\otimes \, n}, V_A)$ and 
\eqref{MC-beta-1n} is satisfied for every $X \in (\Ger^{\vee}(m) \otimes V_A^{\otimes\, m})_{S_m}$ 
with $m \le n$, equation \eqref{MC-beta-n} is also satisfied for every 
$X \in (\Ger^{\vee}(m) \otimes V_A^{\otimes\, m})_{S_m}$ with $m \le n$. 

For $X \in (\Ger^{\vee}(n+1) \otimes V_A^{\otimes\, (n+1)})_{S_{n+1}}$\,, equation \eqref{MC-beta-n}
can be rewritten as 
\begin{equation}
\label{MC-beta-n-new}
- \beta_n \bullet \al (X) + \al \bullet \beta_n (X) = - \sum_{k \ge 2} \frac{1}{k!}  
\{ \beta^{(n-1)}, \dots, \beta^{(n-1)} \}_{k} (X)\,. 
\end{equation}

Let us denote by $\ga$ the element in $\Hom_{S_{n+1}} (\Ger^{\vee}(n+1) \otimes V_A^{\otimes \, (n+1)}, V_A)$
defined as 
\begin{equation}
\label{ga-dfn}
\ga : =  \sum_{k \ge 2} \frac{1}{k!}  
\{ \beta^{(n-1)}, \dots, \beta^{(n-1)} \}_{k} \Big|_{  \Ger^{\vee}(n+1) \otimes V_A^{\otimes\, (n+1)}  }
\end{equation}

Evaluating the Bianchi type identity \cite[Lemma 4.5]{Getzler} 
\begin{multline}
\label{Bianchi}
\sum_{k \ge 2} \frac{1}{k!} d_{\al} \{ \beta^{(n-1)}, \dots, \beta^{(n-1)} \}_{k} 
+ \sum_{k \ge 1} \frac{1}{k!} \{ \beta^{(n-1)}, \dots, \beta^{(n-1)},  d_{\al} \beta^{(n-1) } \}_{k+1} \\
+ \sum_{\substack{ k \ge 2 \\  t \ge 1} } 
\frac{1}{k! t!} \{ \beta^{(n-1)}, \dots, \beta^{(n-1)}, \{ \beta^{(n-1)}, \dots, \beta^{(n-1)} \}_k \}_{t+1} = 0
\end{multline}
on an arbitrary element 
$$
Y \in (\Ger^{\vee}(n+2) \otimes V_A^{\otimes\, (n+2)})_{S_{n+2}}
$$
and using the fact that 
$$
\beta^{(n-1)} (X) = 0\,, \qquad \forall ~~ X \in (\Ger^{\vee}(m) \otimes V_A^{\otimes\, m})_{S_m}
~~\textrm{with}~~ m \ge n
$$
we deduce that element $\ga$ \eqref{ga-dfn} is a cocycle in cochain 
complex \eqref{Conv-End-VA} with differential \eqref{diff-Conv}. 

Thus Theorem \ref{thm:rigidity} implies that equation 
\eqref{MC-beta-n-new} can always be solved for $\beta_n$. 
 
This inductive argument concludes the proof of Corollary \ref{cor:VA-VAQ}. 
\end{proof}

\subsection{The Gerstenhaber algebra $V_A$ is intrinsically formal}
\label{app:int-formal}

Let $(\Cbu, \md)$ be an arbitrary cochain complex whose cohomology 
is isomorphic to $V_A$
\begin{equation}
\label{H-Cbu-VA}
H^{\bul}(\Cbu) \cong V_A\,. 
\end{equation}

Let us consider $V_A$ as the cochain complex with the 
zero differential and choose\footnote{Such a quasi-isomorphism exists
since we are dealing with cochain complexes of vector spaces over a field.} 
a quasi-isomorphism of cochain complexes 
\begin{equation}
\label{VA-to-Cbu}
I : V_A \to \Cbu\,.
\end{equation}
 
Let us assume that $\Cbu$ carries a $\Ger_{\infty}$-structure such 
that the map $I$ induces an isomorphism of Gerstenhaber algebras 
$V_A \cong H^{\bul}(\Cbu)$\,. 

Then Theorem \ref{thm:rigidity} gives us the following remarkable corollary:
\begin{cor}
\label{cor:exists!}
There exists a $\Ger_{\infty}$-morphism 
\begin{equation}
\label{Ger-VA-Cbu}
U : V_A \leadsto \Cbu
\end{equation}
whose linear term coincides with $I$ \eqref{VA-to-Cbu}. 
Moreover, any two such $\Ger_{\infty}$-morphisms
\begin{equation}
\label{two-guys}
U, ~ \wt{U} ~: ~ V_A \leadsto \Cbu
\end{equation}
are homotopy equivalent. 
\end{cor}
\begin{remark}
\label{rem:Hinich}
The above statement is a slight refinement of one proved 
in \cite[Section 5]{Hinich}. Following V. Hinich, we say that 
the Gerstenhaber algebra $V_A$ is intrinsically formal. 
\end{remark}~\\

\begin{proof}[~of Corollary \ref{cor:exists!}]
By the Homotopy Transfer Theorem \cite[Section 5]{DHR}, \cite[Section 10.3]{LV-book}, 
there exists a $\Ger_{\infty}$-structure $Q$ on $V_A$ and 
a $\Ger_{\infty}$-quasi-isomorphism 
\begin{equation}
\label{U-pr}
U'  : V_A^Q \leadsto \Cbu\,, 
\end{equation}
such that 
\begin{itemize}

\item the binary operations of the $\Ger_{\infty}$-structure $Q$ on $V_A$ are the 
Schouten bracket and the usual multiplication of polyvector fields,

\item the linear term of $U'$ coincides with $I$. 

\end{itemize}

Corollary \ref{cor:VA-VAQ} implies that there exists a $\Ger_{\infty}$-morphism 
\begin{equation}
\label{U-corr}
U_{\corr} : V_A \leadsto  V_A^Q\,,
\end{equation}
whose linear term is the identity map $\id: V_A \to V_A$\,. 

Hence the composition
\begin{equation}
\label{U}
U  = U'  \circ U_{\corr} :  V_A \leadsto \Cbu 
\end{equation}
is a desired $\Ger_{\infty}$-morphism.

To prove the second claim, we need the $\LLie$-algebra 
\begin{equation}
\label{Map-VA-Cbu}
\Hom(\Ger^{\vee}(V_A), \Cbu)
\end{equation} 
corresponding to the Gerstenhaber algebra $V_A$ and 
the $\Ger_{\infty}$-algebra $\Cbu$\,. The differential $\cD$ 
on \eqref{Map-VA-Cbu} is given by the formula
\begin{equation}
\label{diff-Map-VA-Cbu}
\cD (\Psi) : = \md \circ \Psi - (-1)^{|\Psi|} \Psi \circ Q_{\wedge, \{~ ,~\}}~ ,
\qquad \Psi \in \Hom(\Ger^{\vee}(V_A), \Cbu)\,,
\end{equation}
where $\md$ is the differential on $\Cbu$ and $Q_{\wedge, \{~ ,~\}}$ is 
the differential on the $\Ger^{\vee}$-coalgebra $\Ger^{\vee}(V_A)$ corresponding 
to the standard Gerstenhaber structure on $V_A$. 

The multi-brackets $\{~, ~, \dots, ~\}_m$ are defined by the general formula 
(see eq. \eqref{m-bracket}) in terms of the $\Ger^{\vee}$-coalgebra structure on 
$\Ger^{\vee}(V_A)$ and the $\Ger_{\infty}$-structure on $\Cbu$. 

Let us recall (see Appendix \ref{app:conv} for more details) 
that $\Ger_{\infty}$-morphisms from $V_A$ to $\Cbu$ are in bijection with 
 Maurer-Cartan elements of $\La^{-1}\Lie_{\infty}$-algebra \eqref{Map-VA-Cbu}
and $\Ger_{\infty}$-morphisms \eqref{two-guys} are homotopy equivalent 
if and only if the corresponding  Maurer-Cartan elements $P$ and $\wt{P}$ in \eqref{Map-VA-Cbu}
are isomorphic $0$-cells in the Deligne-Getzler-Hinich $\infty$-groupoid \cite{Getzler} 
of \eqref{Map-VA-Cbu}.  

So our goal is to prove that any two  Maurer-Cartan elements 
$P$ and $\wt{P}$ in \eqref{Map-VA-Cbu} satisfying 
\begin{equation}
\label{P-wtP-VA}
P \Big|_{V_A} = \wt{P} \Big|_{V_A} = I ~:~ V_A ~\to~ \Cbu 
\end{equation}
are isomorphic.  

Condition \eqref{P-wtP-VA} implies that 
$$
\wt{P} - P \in  \cF_2 \Hom(\Ger^{\vee}(V_A), \Cbu)\,,
$$ 
where $\cF_{\bul} \Hom(\Ger^{\vee}(V_A), \Cbu)$ is the arity 
filtration \eqref{filtr-Map-VW} on $\Hom(\Ger^{\vee}(V_A), \Cbu)$\,.

Let us assume that we constructed a sequence of 
 Maurer-Cartan elements 
\begin{equation}
\label{PPP}
P=P_2, P_3, P_4, \dots, P_{n+1}
\end{equation}
such that for every $2 \le m \le n+1$
\begin{equation}
\label{wtP-Pm}
\wt{P} - P_m \in \cF_m \Hom(\Ger^{\vee}(V_A), \Cbu)
\end{equation} 
and for every $2 \le m \le n$ there exists $1$-cell 
$$ 
P'_m(t) + d t\, \xi_{m-1}  \in  \Hom(\Ger^{\vee}(V_A), \Cbu) \hotimes \Omb(\sfDel_1)
$$
which connects $P_{m}$ to $P_{m+1}$ and such that 
\begin{equation}
\label{xi-1m-cF}
\xi_{m-1} \in \cF_{m-1}  \Hom(\Ger^{\vee}(V_A), \Cbu)\,, 
\end{equation}
and
\begin{equation}
\label{Ppr-P-cF}
P'_m(t)- P_m  \in  \cF_m \Hom(\Ger^{\vee}(V_A), \Cbu)   \hotimes \bbK[t]\,.
\end{equation}

Let us now prove that one can construct a $1$-cell 
\begin{equation}
\label{P-pr-xi-n}
P'_{n+1}(t) + d t\, \xi_n \in  \Hom(\Ger^{\vee}(V_A), \Cbu) \hotimes \Omb(\sfDel_1)
\end{equation}
such that 
$$
P'_{n+1}(t) \Big|_{t = 0} = P_{n+1}\,,
$$
\begin{equation}
\label{xi-n-cF}
\xi_n \in \cF_{n}  \Hom(\Ger^{\vee}(V_A), \Cbu)\,, 
\end{equation} 
\begin{equation}
\label{Ppr-P-cF-n1}
P'_{n+1}(t) - P_{n+1} \in  \cF_{n+1} \Hom(\Ger^{\vee}(V_A), \Cbu) \hotimes \bbK[t]\,,
\end{equation}
and the  Maurer-Cartan element
\begin{equation}
\label{P-n2-dfn}
P_{n+2} : = P'_{n+1}(t)  \Big|_{t = 1}
\end{equation}
satisfies the condition 
\begin{equation}
\label{wtP-Pn2}
\wt{P} - P_{n+2} \in \cF_{n+2} \Hom(\Ger^{\vee}(V_A), \Cbu)\,.
\end{equation} 

Let us denote the difference $\wt{P}- P_{n+1}$ by $K$. 
Since $\wt{P} - P_{n+1} \in \cF_{n+1} \Hom(\Ger^{\vee}(V_A), \Cbu)$,
\begin{equation}
\label{K}
K = \sum_{m \ge n+1} K_m\,, \qquad K_m \in \Hom_{S_m}(\Ger^{\vee}(m) \otimes V_A^{\otimes\, m}, \Cbu)\,.
\end{equation}

Subtracting the left hand side of the  Maurer-Cartan equation
\begin{equation}
\label{MC-P-n1}
\cD (P_{n+1}) + \sum_{m \ge 2} \frac{1}{m!} \{P_{n+1}, P_{n+1}, \dots, P_{n+1}\}_m   = 0
\end{equation}
from the left hand side of the  Maurer-Cartan equation
\begin{equation}
\label{MC-wt-P}
\cD (\wt{P}) + \sum_{m \ge 2} \frac{1}{m!} \{\wt{P}, \wt{P}, \dots, \wt{P}\}_m   = 0
\end{equation}
we see that element \eqref{K} satisfies the equation
\begin{equation}
\label{MC-K}
\cD(K) + \sum_{m \ge 1} \frac{1}{m!}\{P_{n+1}, \dots, P_{n+1}, K \}_{m+1} 
+ \sum_{m \ge 2} \frac{1}{m!} \{K, K, \dots, K\}^{P_{n+1}}_m   = 0\,,
\end{equation}
where the multi-bracket $ \{K, K, \dots, K\}^{P_{n+1}}_m$ is defined by 
the formula
\begin{equation}
\label{brack-twist-Pn1}
\{X_1 , X_2, \dots, X_m \}^{P_{n+1}}_m : = 
\sum_{q \ge 0} \frac{1}{q!} \{P_{n+1}, \dots, P_{n+1}, X_1 , X_2, \dots, X_m \}_{q+m}
\end{equation}

Evaluating \eqref{MC-K} on $\Ger^{\vee}(n+1) \otimes V_A^{\otimes\, (n+1)}$
and using the fact that 
\begin{equation}
\label{K-cFn1}
K \in \cF_{n+1} \Hom(\Ger^{\vee}(V_A), \Cbu)\,,
\end{equation}
we conclude that 
\begin{equation}
\label{K-n1-cocycle}
\md \circ K_{n+1} = 0\,,
\end{equation}
where $\md$ is the differential on $\Cbu$.  

Hence there exist elements
$$
K^{V_A}_{n+1} \in  \Hom_{S_{n+1}}(\Ger^{\vee}(n+1) \otimes V_A^{\otimes\, (n+1)}, V_A)
$$
and
$$
K'_{n+1} \in \Hom_{S_{n+1}}(\Ger^{\vee}(n+1) \otimes V_A^{\otimes\, (n+1)}, \Cbu)
$$
such that 
\begin{equation}
\label{K-Kpr}
K_{n+1}  = I \circ K^{V_A}_{n+1} + \md \circ  K'_{n+1}\,.
\end{equation}

Next, evaluating \eqref{MC-K} on  $Y \in \Ger^{\vee}(n+2) \otimes V_A^{\otimes\, (n+2)}$
and using inclusion \eqref{K-cFn1} again, we get the following identity 
\begin{equation}
\label{level-n2}
\md \circ K_{n+2} (Y) - K_{n+1} \circ Q_{\wedge, \{~ ,~\}} (Y) + \{P_{n+1}, K_{n+1}\}_2 (Y) = 0\,.  
\end{equation}

Unfolding $ \{P_{n+1}, K_{n+1}\}_2 (Y)$ we get 
\begin{equation}
\label{the-worst-term}
 \{P_{n+1}, K_{n+1}\}_2 (Y) =  \sum_{i=1}^{n+2} Q_{\Cbu}\Big(
(\id_{\Ger^{\vee}(2)} \otimes K_{n+1} \otimes I) \circ 
\big( \D_{\bt_i} \otimes \id^{\otimes \, (n+2)} \big) (Y) \Big)\,, 
\end{equation}
where $Q_{\Cbu}$ is the $\Ger_{\infty}$-structure on $\Cbu$, 
$\bt_i$ is the $(n+2)$-labeled planar tree shown on figure \eqref{fig:bt-i}, 
and $\D_{\bt_i}$ is the corresponding component of the comultiplication
\begin{equation}
\label{Delta-bt-i}
\D_{\bt_i} : \Ger^{\vee}(n+2) \to \Ger^{\vee}(2) \otimes \Ger^{\vee}(n+1)\,.   
\end{equation}
\begin{figure}[htp]
\centering
\begin{tikzpicture}[scale=0.5, >=stealth']
\tikzstyle{w}=[circle, draw, minimum size=3, inner sep=1]
\tikzstyle{b}=[circle, draw, fill, minimum size=3, inner sep=1]
\node [b] (l1) at (-2,6) {};
\draw (-2,6.6) node[anchor=center] {{\small $1$}};
\node [b] (l2) at (0,6) {};
\draw (0,6.6) node[anchor=center] {{\small $2$}};
\draw (2,6) node[anchor=center] {{$\dots$}};
\node [b] (l1i) at (4,6) {};
\draw (4,6.6) node[anchor=center] {{\small $i-1$}};
\node [b] (li1) at (6,6) {};
\draw (6,6.6) node[anchor=center] {{\small $i+1$}};
\draw (8,6) node[anchor=center] {{$\dots$}};
\node [b] (ln2) at (10,6) {};
\draw (10,6.6) node[anchor=center] {{\small $n+2$}};
\node [b] (li) at (10,3) {};
\draw (10,3.6) node[anchor=center] {{\small $i$}};
\node [w] (v2) at (4,3) {};
\node [w] (v1) at (7,1) {};
\node [b] (r) at (7,0) {};
\draw (v2) edge (l1);
\draw (v2) edge (l2);
\draw (v2) edge (l1i);
\draw (v2) edge (li1);
\draw (v2) edge (ln2);
\draw (v1) edge (v2);
\draw (v1) edge (li);
\draw (r) edge (v1);
\end{tikzpicture}
\caption{\label{fig:bt-i} The $(n+2)$-labeled planar tree $\bt_i$}
\end{figure}

Now using \eqref{K-Kpr} and \eqref{the-worst-term}, we rewrite \eqref{level-n2}
as follows 
\begin{multline}
\label{level-n2-unfold}
\md \circ K_{n+2} (Y)  - I \circ (K^{V_A}_{n+1} \bullet \al)(Y) \\
+ \sum_{i=1}^{n+2} Q_{\Cbu}\Big(
(\id_{\Ger^{\vee}(2)} \otimes  (\md \circ  K'_{n+1}) \otimes I) \circ 
\big( \D_{\bt_i} \otimes \id^{\otimes \, (n+2)} \big) (Y) \Big)    \\
+ \sum_{i=1}^{n+2} Q_{\Cbu}\Big(
(\id_{\Ger^{\vee}(2)} \otimes (I \circ K^{V_A}_{n+1})  \otimes I) \circ 
\big( \D_{\bt_i} \otimes \id^{\otimes \, (n+2)} \big) (Y) \Big)
 = 0\,,
\end{multline}
where $\al$ is defined in \eqref{al}.

Since the last two sums in \eqref{level-n2-unfold} involve only binary 
$\Ger_{\infty}$-operations on $\Cbu$ and these binary operations 
induce the usual multiplication and the Schouten bracket on $V_A$, 
we conclude that each term in the first sum in  \eqref{level-n2-unfold} is 
$\md$-exact and the second sum in  \eqref{level-n2-unfold} is cohomologous to 
$$
I \circ (\al \bullet K^{V_A}_{n+1}) (Y)
$$

Therefore, identity \eqref{level-n2-unfold} implies that for every  
$Y \in \Ger^{\vee}(n+2) \otimes V_A^{\otimes\, (n+2)}$ the expression 
$$
I \circ (\al \bullet K^{V_A}_{n+1} -  K^{V_A}_{n+1} \bullet \al ) (Y)
$$
is $\md$-exact. Thus 
$$
\al \bullet K^{V_A}_{n+1} -  K^{V_A}_{n+1} \bullet \al  = 0
$$
or, in other words, the element $K^{V_A}_{n+1}$ is a cocycle in complex \eqref{Conv-End-VA}
with differential \eqref{diff-Conv}. 

Hence, by Theorem \ref{thm:rigidity}, there exists a degree $-1$ element 
\begin{equation}
\label{wt-K-n}
\wt{K}^{V_A}_{n}  \in  \Hom_{S_{n}}(\Ger^{\vee}(n) \otimes V_A^{\otimes\, (n)}, V_A)
\end{equation}
such that 
\begin{equation}
\label{wt-K-KVA}
K^{V_A}_{n+1}  = [\al, \wt{K}^{V_A}_{n} ]\,.
\end{equation}

Let us now consider the degree $-1$ element 
\begin{equation}
\label{xi-n}
\xi_n =  I \circ \wt{K}^{V_A}_{n}  + K''_{n+1} \in  \cF_{n} \Hom(\Ger^{\vee}(V_A), \Cbu)\,,
\end{equation}
where $\wt{K}^{V_A}_{n}$ is element \eqref{wt-K-n} entering 
equation \eqref{wt-K-KVA} and $K''_{n+1}$ is an element in 
$$
\Hom_{S_{n+1}} \big(  \Ger^{\vee}(n+1) \otimes V_A^{\otimes\, (n+1)} , \Cbu \big)
$$
which will be determined later. 

Using $\xi_n$, we define $ P'_{n+1}(t) \in  \Hom(\Ger^{\vee}(V_A), \Cbu) \hotimes \bbK[t] $ as the 
limiting element of the recursive procedure
\begin{multline}
\label{iter-P-pr}
(P')^{(0)} : = P_{n+1}\,, \\ (P')^{(k+1)}(t) : = 
P_{n+1} + \int_{0}^t d t_1 \Big( \cD (\xi_n) + \sum_{m \ge 1} \frac{1}{m!} 
\{(P')^{(k)}(t_1), \dots, (P')^{(k)}(t_1), \xi_n \}_{m+1} \Big)\,.
\end{multline}

Since 
$$
\md \big( I \circ \wt{K}^{V_A}_{n} \big) = 0 
$$
the element $\xi_n$ satisfies the condition 
$$
\cD(\xi_n) \in \cF_{n+1} \Hom(\Ger^{\vee}(V_A), \Cbu)\,.
$$
Hence, by Lemma \ref{lem:isom-MC}, the sum 
\begin{equation}
\label{1-cell-one}
P'_{n+1}(t) + d t \xi_n \in \Hom(\Ger^{\vee}(V_A), \Cbu) \hotimes \Omb(\sfDel_1)
\end{equation}
is a $1$-cell in the $\infty$-groupoid corresponding to $\Hom(\Ger^{\vee}(V_A), \Cbu)$
satisfying \eqref{Ppr-P-cF-n1} and such that the Maurer-Cartan element $P_{n+2}$ \eqref{P-n2-dfn}
satisfies the condition 
\begin{equation}
\label{P-n2-good}
P_{n+2} - P_{n+1} - \cD(\xi_n) - \{P_{n+1}, \xi_n\}_2 \in \cF_{n+2} \Hom(\Ger^{\vee}(V_A), \Cbu)\,.
\end{equation}

Let us now show that, by choosing the element $K''_{n+1}$ in \eqref{xi-n} appropriately, 
we can get desired inclusion \eqref{wtP-Pn2}.  
 
For this purpose we unfold $\{P_{n+1}, \xi_n\}_2 (Y)$ for an arbitrary $Y \in  
\Ger^{\vee}(n+1) \otimes V_A^{\otimes\, (n+1)}$ and get 
\begin{equation}
\label{unpl-term}
\{P_{n+1}, \xi_n\}_2 (Y) = 
 \sum_{i=1}^{n+1} Q_{\Cbu}\Big(
(\id_{\Ger^{\vee}(2)} \otimes (I \circ \wt{K}^{V_A}_{n}) \otimes I) \circ 
\big( \D_{\bt'_i} \otimes \id^{\otimes \, (n+1)} \big) (Y) \Big)\,, 
\end{equation}
where $Q_{\Cbu}$ is the $\Ger_{\infty}$-structure on $\Cbu$, 
$\bt'_i$ is the $(n+1)$-labeled planar tree shown on figure \eqref{Delta-bt-pr-i}, 
and $\D_{\bt'_i}$ is the corresponding component of the comultiplication
\begin{equation}
\label{Delta-bt-pr-i}
\D_{\bt'_i} : \Ger^{\vee}(n+1) \to \Ger^{\vee}(2) \otimes \Ger^{\vee}(n)\,.   
\end{equation}
\begin{figure}[htp]
\centering
\begin{tikzpicture}[scale=0.5, >=stealth']
\tikzstyle{w}=[circle, draw, minimum size=3, inner sep=1]
\tikzstyle{b}=[circle, draw, fill, minimum size=3, inner sep=1]
\node [b] (l1) at (-2,6) {};
\draw (-2,6.6) node[anchor=center] {{\small $1$}};
\node [b] (l2) at (0,6) {};
\draw (0,6.6) node[anchor=center] {{\small $2$}};
\draw (2,6) node[anchor=center] {{$\dots$}};
\node [b] (l1i) at (4,6) {};
\draw (4,6.6) node[anchor=center] {{\small $i-1$}};
\node [b] (li1) at (6,6) {};
\draw (6,6.6) node[anchor=center] {{\small $i+1$}};
\draw (8,6) node[anchor=center] {{$\dots$}};
\node [b] (ln1) at (10,6) {};
\draw (10,6.6) node[anchor=center] {{\small $n+1$}};
\node [b] (li) at (10,3) {};
\draw (10,3.6) node[anchor=center] {{\small $i$}};
\node [w] (v2) at (4,3) {};
\node [w] (v1) at (7,1) {};
\node [b] (r) at (7,0) {};
\draw (v2) edge (l1);
\draw (v2) edge (l2);
\draw (v2) edge (l1i);
\draw (v2) edge (li1);
\draw (v2) edge (ln1);
\draw (v1) edge (v2);
\draw (v1) edge (li);
\draw (r) edge (v1);
\end{tikzpicture}
\caption{\label{fig:bt-pr-i} The $(n+1)$-labeled planar tree $\bt'_i$}
\end{figure}

Since the right hand side of \eqref{unpl-term} involves only 
binary $\Ger_{\infty}$-operations on $\Cbu$ and these binary operations 
induce the usual multiplication and the Schouten bracket on $V_A$, 
we conclude that $\{P_{n+1}, \xi_n\}_2 (Y)$ is cohomologous (in $\Cbu$) to 
$$
I \circ (\al \bullet \wt{K}^{V_A}_{n}) (Y)\,,
$$
where $\al$ is defined in \eqref{al}.

In other words, there exists an element 
\begin{equation}
\label{phi}
\phi \in  \Hom_{S_{n+1}} \big(  \Ger^{\vee}(n+1) \otimes V_A^{\otimes\, (n+1)} , \Cbu \big)
\end{equation}
such that 
$$
\{P_{n+1}, \xi_n\}_2 (Y) = I \circ (\al \bullet \wt{K}^{V_A}_{n}) (Y) + \md \circ \phi(Y). 
$$

Hence the expression  $\big( \cD(\xi_n) + \{P_{n+1}, \xi_n\}_2 \big)(Y)$ can be rewritten 
as
\begin{equation}
\label{key-expression}
\big( \cD(\xi_n) + \{P_{n+1}, \xi_n\}_2 \big)(Y) = \md \circ K''_{n+1}  (Y) + 
\md \circ \phi(Y) +  I \circ [\al, \wt{K}^{V_A}_{n}] (Y)\,. 
\end{equation}

Thus if 
$$
K''_{n+1}  =  K'_{n+1} - \phi 
$$ 
then equations \eqref{K-Kpr}, \eqref{wt-K-KVA}, and inclusion \eqref{P-n2-good} imply  
that \eqref{wtP-Pn2} holds, as desired. 

Thus we showed that one can construct an infinite sequence 
of  Maurer-Cartan elements 
$$
P=P_2, P_3, P_4, \dots
$$
and an infinite sequence of $1$-cells $(m \ge 2)$
\begin{equation}
\label{1-cell-m}
P'_m(t) + d t\, \xi_{m-1}  \in  \Hom(\Ger^{\vee}(V_A), \Cbu) \hotimes \Omb(\sfDel_1)
\end{equation}
such that for every $m \ge 2$
$$
\wt{P} - P_m \in \cF_m \Hom(\Ger^{\vee}(V_A), \Cbu)\,,
$$
the $1$-cell $P'_m(t) + d t\, \xi_{m-1}$ connects $P_{m}$ to $P_{m+1}$ 
\begin{equation}
\label{needed-1}
\xi_{m-1} \in \cF_{m-1}  \Hom(\Ger^{\vee}(V_A), \Cbu)\,, 
\end{equation}
and
\begin{equation}
\label{needed-11}
P'_m(t)- P_m  \in  \cF_m \Hom(\Ger^{\vee}(V_A), \Cbu)  \hotimes \bbK[t]\,.
\end{equation}

Since the $\LLie$-algebra  $\Hom(\Ger^{\vee}(V_A), \Cbu)$ is complete with 
respect to ``arity'' filtration \eqref{filtr-Map-VW}, inclusions \eqref{needed-1}
and \eqref{needed-11} imply that we can form the infinite 
composition\footnote{Note that the composition of $1$-cells in an infinity groupoid is 
not unique but this does not create a problem.} of all $1$-cells \eqref{1-cell-m}
and get a $1$-cell which connects the Maurer-Cartan element $P=P_2$ to 
the Maurer-Cartan element $\wt{P}$. 
 
Corollary \ref{cor:exists!} is proved.
\end{proof}

\section{On derivations of $\Cyl(\La^2 \coCom)$}
\label{app:derivations}

Let $\cC$ be a coaugmented cooperad in the category of graded vector spaces 
and $\cC_{\c}$ be the cokernel of the coaugmentation. As above, we assume that 
$\cC(0) = \bfzero$ and $\cC(1) = \bbK$. 

Following \cite[Section 3]{Action}, \cite{Fresse}, we will denote by 
$\Cyl(\cC)$ the $2$-colored dg operad
whose algebras are pairs $(V,W)$ with the data 
\begin{enumerate}
\item a $\Cobar(\cC)$-algebra structure on $V$, 

\item a $\Cobar(\cC)$-algebra  structure on $W$, and 

\item an $\infty$-morphism $F$ from $V$ to $W$, i.e. 
a homomorphism of corresponding dg $\cC$-coalgebras 
$\cC(V) \to \cC(W)$.
  
\end{enumerate}

In fact, if we forget about the differential, then $\Cyl(\cC)$ is a free operad on a certain 
$2$-colored collection $\cM(\cC)$ naturally associated to $\cC$. 

Following the conventions of Section \ref{sec:grtact}, we denote by 
\begin{equation}
\label{Der-pr-Cyl-cC}
\Der' \big(\Cyl(\cC) \big) 
\end{equation}
the dg Lie algebra of derivations $\cD$ of $\Cyl(\cC)$ subject to the 
condition
\begin{equation}
\label{prime-cond-app}
p\circ \cD = 0\,,
\end{equation}
where $p$ is the canonical projection from $\Cyl(\cC)$ onto $\cM(\cC)$. 

We have the following generalization of \eqref{holes}:
\begin{prop}
\label{prop:morph-Lie-OK}
The dg Lie algebra $\Der' \big(\Cyl( \La^2 \coCom) \big)$ does not have 
non-zero elements in degrees $\le 0$, i.e. 
$$
\Der' \big(\Cyl(\La^2 \coCom) \big)^{\, \le 0} = \bfzero\,.
$$
\end{prop}
\begin{proof}
Let us denote by $\al$ and $\beta$, respectively, the first and the second color 
for the collection $\cM(\La^2 \coCom)$ and the operad $\Cyl(\La^2 \coCom)$.

Recall from \cite{Action} that $\Cyl(\La^2 \coCom)$ is 
generated by the collection $\cM = \cM(\La^2 \coCom)$ with
\begin{align*}
\cM(n,0;\alpha) = \susp \La^2 \coCom_\circ(n)  = \bs^{3-2n} \bbK \,, \\
\cM(0,n;\beta) = \susp\La^2 \coCom_\circ(n) =  \bs^{3-2n} \bbK \,, \\
\cM(n,0;\beta) = \La^2 \coCom(n) =  \bs^{2-2n} \bbK \,,
\end{align*}
and with all the remaining spaces being zero. Let $\cD$ be a derivation of $\Cyl(\La^2 \coCom)$ of degree $\le 0$. 

Since
$$
\Cyl \big(\La^2 \coCom \big)(n,0, \al) = \La\Lie_{\infty}(n) \qquad 
\textrm{and} \qquad 
\Cyl \big(\La^2 \coCom \big)(0,n, \beta) = \La\Lie_{\infty}(n)\,,
$$
observation \eqref{holes} implies that  
$$
\cD \Big|_{\cM(n,0; \alpha)} ~ = ~ 
\cD \Big|_{\cM(0,n; \beta)} ~ = ~ 0\,.
$$

Hence, it suffices to show that 
\begin{equation}
\label{cD-on-mixed}
\cD \Big|_{ \cM(n,0;\beta) }  ~ = ~ 0\,. 
\end{equation}

Let us denote by $ \pi_0(\Tree_k(n))$ the set of isomorphism classes 
of labeled $2$-colored planar trees corresponding to corolla $(n,0; \beta)$ 
with $k$ internal vertices. Figure \ref{fig:exam}  show two 
examples of such trees with $n=5$ leaves. The left tree has 
$k=2$ internal vertices and the right tree has $k=3$ internal 
vertices. 
\begin{figure}[htp]
\centering
\begin{minipage}[t]{0.45\linewidth} 
\centering 
\begin{tikzpicture}[scale=0.5, >=stealth']
\tikzstyle{w}=[circle, draw, minimum size=3, inner sep=1]
\tikzstyle{b}=[circle, draw, fill, minimum size=3, inner sep=1]
\node [b] (al3) at (3,6) {};
\draw (3,6.6) node[anchor=center] {{\small $3$}};
\node [b] (al1) at (5,6) {};
\draw (5, 6.6) node[anchor=center] {{\small $1$}};
\node [b] (al4) at (7,6) {};
\draw (7, 6.6) node[anchor=center] {{\small $4$}};
\node [w] (v2) at (5,4) {};
\node [b] (al5) at (3,4) {};
\draw (3,4.6) node[anchor=center] {{\small $5$}};
\node [b] (al2) at (1,4) {};
\draw (1,4.6) node[anchor=center] {{\small $2$}};
\node [w] (v1) at (3,2) {};
\node [b] (r) at (3,0) {};
\draw (v2) edge (al3);
\draw (v2) edge (al1);
\draw (v2) edge (al4);
\draw (v1) edge (al2);
\draw (v1) edge (al5);
\draw (v1) edge (v2);
\draw [dashed] (r) edge (v1);
\end{tikzpicture}
\end{minipage} ~
\begin{minipage}[t]{0.45\linewidth} 
\centering 
\begin{tikzpicture}[scale=0.5, >=stealth']
\tikzstyle{w}=[circle, draw, minimum size=3, inner sep=1]
\tikzstyle{b}=[circle, draw, fill, minimum size=3, inner sep=1]
\node [b] (al4) at (0,6) {};
\draw (0,6.6) node[anchor=center] {{\small $4$}};
\node [b] (al3) at (2,6) {};
\draw (2,6.6) node[anchor=center] {{\small $3$}};
\node [b] (al1) at (4,6) {};
\draw (4,6.6) node[anchor=center] {{\small $1$}};
\node [b] (al5) at (6,6) {};
\draw (6,6.6) node[anchor=center] {{\small $5$}};
\node [b] (al2) at (8,6) {};
\draw (8,6.6) node[anchor=center] {{\small $2$}};
\node [w] (v2) at (1,4) {};
\node [w] (v3) at (6,4) {};
\node [w] (v1) at (3.5,2) {};
\node [b] (r) at (3.5,0) {};
\draw (v2) edge (al4);
\draw (v2) edge (al3);
\draw (v3) edge (al1);
\draw (v3) edge (al5);
\draw (v3) edge (al2);
\draw [dashed] (v1) edge (v2);
\draw [dashed] (v1) edge (v3);
\draw [dashed] (r) edge (v1);
\end{tikzpicture}
\end{minipage}
\caption{\label{fig:exam} Solid edges carry the color $\al$ and 
dashed edges carry the color $\beta$; internal vertices are denoted by 
small white circles; leaves and the root vertex are denoted by small black 
circles}
\end{figure}
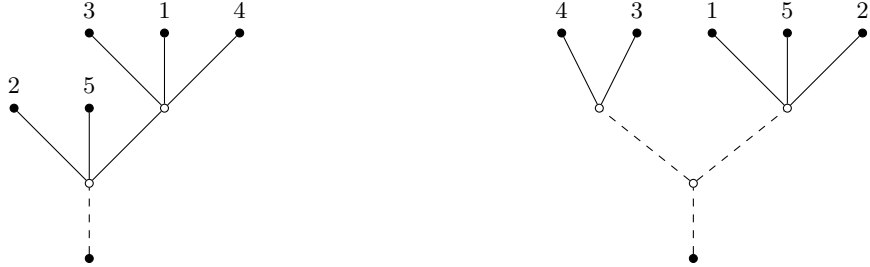

For a generator $X \in  \cM(n,0;\beta) = \susp^{2-2n} \bbK$, the element 
$\cD(X) \in \Cyl(\La^2\coCom)$ takes the form
\begin{equation}
\label{cD-X}
\cD(X) ~ = ~ \sum_{k \ge 2} ~
\sum_{z \in \pi_0(\Tree_k(n))} (\bt_z; X_1, ..., X_k)
\end{equation}
where $\bt_z$ is a representative of an isomorphism class $z \in \pi_0(\Tree_k(n))$ and 
$X_i$ are the corresponding elements of $\cM$. 
 
For every term in sum \eqref{cD-X}, we have $k_1$  $X_i$'s  
in $\susp \La^2 \coCom_\circ $ (call them $X_{i_a}$), and $k_2$  $X_i$'s in $\La^2\coCom$ 
(call them $X_{j_b}$). 

We obviously have that $k=k_1 + k_2$ and 
\begin{equation}
\label{deg-cD}
|\cD| = \sum_{a=1}^{k_1} |X_{i_a}|   + \sum_{b=1}^{k_2} |X_{j_b}| - |X|
\end{equation}
or equivalently 
$$
|\cD| = 2 (n-1) + \sum_{a=1}^{k_1} (3 - 2 n_{i_a}) + \sum_{b=1}^{k_2} (2 - 2n_{j_b})\,,
$$
where $n_{i_a}$ (resp. $n_{j_b}$) is the number of incoming edges of the vertex corresponding 
to $X_{i_a}$ (resp. $X_{j_b}$)\,.

On the other hand, a simple combinatorics of trees shows that 
$$
n-1 =   \sum_{a=1}^{k_1} (n_{i_a}-1) +  \sum_{b=1}^{k_2} (n_{j_b} - 1)
$$
and hence 
$$
|\cD| = k_1\,.
$$

Since $|\cD| \le 0$ the latter is possible only if $k_1=0 = |\cD|$, i.e. 
every tree in the sum $\cD(X)$ is assembled exclusively from 
mixed colored corollas.  
That would force every tree $\bt$ to have only one internal vertex which contradicts to the fact that 
the summation in \eqref{cD-X} starts at $k = 2$\,.

Therefore \eqref{cD-on-mixed} holds and the proposition follows. 
\end{proof}

~\\

\noindent\textsc{Department of Mathematics,
Temple University, \\
Wachman Hall Rm. 638\\
1805 N. Broad St.,\\
Philadelphia PA, 19122 USA \\
\emph{E-mail addresses:} {\bf vald@temple.edu} and {\bf brian.paljug@temple.edu}} 

\end{document}